\documentclass[11pt]{amsart}
\usepackage{mathrsfs}
\usepackage{amssymb}
\usepackage[T1]{fontenc}
\usepackage{amscd}
\usepackage{amsmath, amssymb}
\usepackage{amsthm}
\usepackage{amsfonts}
\usepackage[colorlinks,linkcolor=blue,citecolor=blue, pdfstartview=FitH]
{hyperref}
\usepackage{backref}
\usepackage{amscd}
\usepackage{easybmat}
\usepackage{mathrsfs}
\usepackage{amsfonts}
\usepackage{color}
\usepackage{pifont}
\usepackage{upgreek}
\usepackage{bm}
\usepackage{hyperref}
\usepackage{shorttoc}
\usepackage{amsmath,amstext,amsthm,a4,amssymb,amscd}
\usepackage[mathscr]{eucal}
\usepackage{mathrsfs}
\usepackage{epsf}

\numberwithin{equation}{section}
\def\Re{{\rm Re}}
\def\Im{{\rm Im}}
\def\p{\partial}
\def\o{\overline}
\def\b{\bar}
\def\mb{\mathbb}
\def\mc{\mathcal}
\def\n{\nabla}

\theoremstyle{plain}
\newtheorem{thm}{Theorem}[section]
\newtheorem{lemma}[thm]{Lemma}

\newtheorem{cor}[thm]{Corollary}
\theoremstyle{definition}
\newtheorem{rem}[thm]{Remark}

\theoremstyle{definition}
\newtheorem{defn}[thm]{Definition}
\newcommand{\comment}[1]{}

\usepackage{fancyhdr}
\newenvironment{aligns}{\equation\aligned}{\endaligned\endequation}

\begin{document}

\title{PLURISUPERHARMONICITY 
 OF  RECIPROCAL ENERGY FUNCTION ON TEICHM\"ULLER SPACE AND WEIL-PETERSSON METRICS}

\makeatletter

\makeatother
\author{InKang Kim}
\author{Xueyuan Wan}
\author{Genkai Zhang}

\address{Inkang Kim: School of Mathematics, KIAS, Heogiro 85, Dongdaemun-gu Seoul, 02455, Republic of Korea}
\email{inkang@kias.re.kr}

\address{Xueyuan Wan: Mathematical Sciences, Chalmers University of Technology and Mathematical Sciences, G\"oteborg University, SE-41296 G\"oteborg, Sweden}
\email{xwan@chalmers.se}

\address{Genkai Zhang: Mathematical Sciences, Chalmers University of Technology and Mathematical Sciences, G\"oteborg University, SE-41296 G\"oteborg, Sweden}
\email{genkai@chalmers.se}
\thanks{Research by Inkang Kim is partially supported by Grant NRF-2017R1A2A2A05001002 and  research by Genkai Zhang is
  partially supported by Swedish
  Research Council (VR), for which we 
  gratefully
 acknowledge}
\begin{abstract}
  We consider harmonic maps
 $u(z):  \mc{X}_z\to N$   
  in a fixed homotopy class
  from
Riemann surfaces 
  $\mc{X}_z$ of genus $g\geq 2$ varying in the Teichm\"u{}ller space $\mathcal T$
  to a Riemannian manifold $N$    with non-positive Hermitian sectional curvature.
  The energy function $E(z)=E(u(z))$ can be
  viewed as a function on $\mathcal T$
  and we study its
  first and the second variations. We prove that the reciprocal energy
  function $E(z)^{-1}$ is  plurisuperharmonic on Teichm\"uller
  space. We also obtain the (strict) plurisubharmonicity of $\log
  E(z)$ and $E(z)$. As an application, we get the following
  relationship between the second variation of logarithmic energy
  function and the
  Weil-Petersson metric if  the harmonic map $u(z)$ is holomorphic or anti-holomorphic and totally geodesic, i.e.,
\begin{align}\label{A1}
\sqrt{-1}\p\b{\p}\log E(z)=\frac{\omega_{WP}}{2\pi(g-1)}.	
\end{align}
We consider also the energy function $E(z)$ associated to the harmonic maps from a fixed compact K\"ahler manifold $M$ to Riemann surfaces $\{\mc{X}_z\}_{z\in\mc{T}}$ in a fixed homotopy class. If $u(z)$ is holomorphic or anti-holomorphic, then (\ref{A1}) is also proved.
 \end{abstract}

\maketitle
\tableofcontents

\section*{Introduction}

Recently, the Weil-Petersson metric and other K\"ahler metrics on Teichm\"uller space
of a surface have been studied extensively.
The Weil-Petersson metric has several interesting properties,
it is K\"ahler \cite{Ahlfors}, incomplete \cite{Chu, Wol4}, geodesically convex \cite{Wol1} and negatively curved \cite{Tromba2, Wol5}, and
the energy function of harmonic maps between Riemann surfaces
is a K\"ahler potential of it \cite{FT, Wolf1}. In this paper we
 consider the log-plurisubharmonicity
 and  the pluri-superharmonicity of reciprocal energy function
 of harmonic maps between Riemann surfaces and a general Riemannian manifold, and we
compare
the  second variation with the Weil-Petersson metric.

Let $\Sigma$ be a Riemann surface
of genus $g\geq 2$
 equipped with  hyperbolic metric, and $M$ and $ N$  Riemannian manifolds.
There are two kinds of harmonic maps whose variations are of interests, the maps
$u: M \to \Sigma$  and maps 
$u: \Sigma \to N$. The primary examples of the first kind are
closed geodesics in $\Sigma$ viewed harmonic maps from the circle to
$\Sigma$. Now as the hyperbolic metric varies in the
Teichm\"u{}ller space $\mathcal T$
we get Riemann surfaces
$\mathcal X_z$ and 
the geodesic length can be viewed
as a function of $z\in \mathcal T$, 
the variation formulas of the geodesic length function, i.e., the energy function,
have been obtained in Axelsson and Schumacher's  formulas \cite{AS0, AS}.
In a recent  paper \cite{KWZ} we find general variational formulas
for harmonic maps $u: M\to \mathcal X_z$
and we  prove the logarithmic plurisubharmonicity of the energy function,
thus generalizing the results in \cite{AS0, AS}.
The another kind of harmonic maps
 $u: \Sigma\to N$  appear also naturally in the study of rigidity \cite
{Toledo} and in  Hitchin
components \cite{Labourie}.
 If $N$ is also a negatively curved
Riemann surface Tromba \cite{Tromba0} showed that this energy function
is strictly plurisubharmonic.
When  $N$ has non-positive Hermitian sectional curvature, Toledo
\cite{Toledo} proved that the energy function is also
plurisubharmonic.
A natural question is whether the logarithm of energy function is also plurisubharmonic. In this paper, we give an affirmative answer to this question. 

 Let $(N, g)$ be a Riemannian manifold with non-positive Hermitian
 sectional curvature (see Definition \ref{HSC}).
  In particular $N$ has non-positive sectional curvature.
 Let $\mc{T}$ be Teichm\"uller space of a  surface of genus
$g\geq 2$, and  $\pi:\mc{X}\to \mc{T}$  Teichm\"uller curve over Teichm\"uller space $\mc{T}$, namely it is the holomorphic family
of Riemann surfaces over $\mc{T}$,  the fiber $\mc{X}_z:=\pi^{-1}(z)$ being exactly the Riemann surface given by the complex structure $z\in\mc{T}$, see e.g. \cite[Section 5]{Ahlfors}.
Let $u_0: (\mc{X}_z, \Phi_z)\to (N,g)$ be a continuous map, where $\Phi_z$ is the hyperbolic  metric on the Riemann surface $\mc{X}_z$. We assume that for  each $z\in \mc{T}$, there is a unique harmonic map $u(z):  (\mc{X}_z, \Phi_z)\to (N, g)$ homotopic to $u_0$.  Then we get a smooth map $u(z,v):\mc{X}\to N$ and the energy
 \begin{align}\label{0.1}
 E(z)=E(u(z))=	\frac{1}{2}\int_{\mc{X}_z} |du(z)|^2 d\mu_{\Phi_z}
\end{align}
is a smooth function on Teichm\"uller space, see \cite{EL, Sampson1, Tromba0} for proofs of smooth dependence in several contexts.  

Our first main theorem is 
\begin{thm}\label{Introthm1}
Let $(N,g)$ be a Riemannian manifold with non-positive Hermitian sectional curvature and fix a smooth map $u_0: \Sigma\to N$. If there is a unique harmonic map $u(z):\mc{X}_z\to N$ in the homotopy class $[u_0]$ for each $z\in\mc{T}$, then the reciprocal energy function $E(z)^{-1}$ is plurisuperharmonic.
\end{thm}
Note that the uniqueness assumption is typically satisfied. For instance, if $(N,g)$ has strictly negative sectional curvature, then the harmonic map is unique unless its image is either a point or a closed geodesic \cite{Hartman}. If $N$ is a locally symmetric space of non-compact type, then $u$ is also unique unless $u_*(\pi_1(\Sigma))$ is centralized by a semi-simple element in the group of isometries of the universal cover of $N$ \cite{Sun}.

We also obtain the strictly plurisubhramonicity of $\log E(z)$. More precisely,
\begin{thm}\label{Introthm2}
Under the conditions of Theorem \ref{Introthm1}, the logarithm of
energy function $\log E(z)$ of $u(z): \mathcal X_z\to N$
is plurisubharmonic. 
Moreover, if $(N,g)$ has strictly negative Hermitian sectional curvature and $d(u(z_0))$ is never zero on $\mc{X}_{z_0}$ for some $z_0\in \mc{T}$, then $\log E(z)$ is strictly plurisubharmonic at $z_0$.
\end{thm}

As a corollary, we obtain the following result of Toledo.
\begin{cor}[{\cite[Theorem 1, 3]{Toledo}}]\label{Introcor3}
Under the conditions of Theorem \ref{Introthm1}, the energy function $E(z)$ is plurisubharmonic. 
Moreover, if $(N,g)$ has strictly negative Hermitian sectional curvature and $d(u(z_0))$ is never zero on $\mc{X}_{z_0}$ for some $z_0\in \mc{T}$, then $E(z)$ is strictly plurisubharmonic at $z_0$.
\end{cor}
The (strict) plurisubharmonicity of
energy function
is proved in  \cite{Toledo}
by using a formula of Micallef-Moore \cite{MM}.
More precisely, let $D$ be a small disk in $\mb{C}$ centered at $0$,
and let
$J=J(s,t)$
be a family of complex structures on $\Sigma$ compatible with the orientation and depending holomorphically on the complex parameter $z=s+\sqrt{-1}t\in D$. Then $E(z)=E(s,t)=E(J(s,t))$, and the complex variation can be obtain from the real variation, i.e.,
\begin{align*}
\Delta E(0)=\frac{\p^2 E}{\p s^2}|_{z=0} +\frac{\p^2 E}{\p t^2}|_{z=0}, 
\end{align*}
where $\Delta=4\p_z\p_{\b{z}}$. The family $J=J(s,t)$
satisfies certain Cauchy-Riemann equation
\cite{Toledo} and the variation can be computed in terms of $J$.
Our method is completely different from Toledo's.
We shall treat the energy function as the push-forward of a differential
form on Teichm\"uller curve $\mc{X}$, by using the canonical
decomposition of the holomorphic cotangent bundle $T^*\mc{X}$, and we
obtain a precise and somewhat more concrete formula on the second variation of energy function. 
 
We proceed to explain further details of our results and methods.
Let
$u(z):=(\mc{X}_z,\Phi_z)\to (N,g)$ be a family of  harmonic maps considered
as a smooth map $u:\mc{X}\to N$.
Let 
$(z;v)=(z^1,\cdots, z^m; v)$
be  local holomorphic coordinates of $\mc{X}$ with $\pi(z,v)=z$, where
$(z)$ denotes the local coordinates of $\mc{T}$ and $(v)$
denotes the local coordinates of Riemann surface $\mc{X}_z$,
$m=3g-3=\dim_{\mathbb C} \mc{T}$. Note that $du\in A^1(\mc{X}, u^*TN)$ can be decomposed as 
\begin{align*}
du=\p u+\b{\p}u\in A^1(\mc{X}, u^*TN),	
\end{align*}
where $\p u$
denotes the $(1,0)$-component of $du$ and $\b{\p}u=\o{\p u}$ denotes the $(0,1)$-component of $du$. Let 
$	\langle \p u\wedge \b{\p}u\rangle$
 denote the (1,1)-form on $\mc{X}$ obtained by combining the wedge
 product in $\mc{X}$ with the Riemannian metric $\langle \cdot, \cdot\rangle$ on $u^*TN$.
Then the energy function $E(z)$ can be expressed as 
\begin{align*}
E(z)=\sqrt{-1}\int_{\mc{X}/\mc{T}}\langle \p u\wedge \b{\p}u\rangle;
\end{align*}
see (\ref{2.1}) below.
Here $\int_{\mc{X}/\mc{T}}$
 denotes the integral along fibers. Then the first and the second variations of energy function are given by 
 \begin{align*}
\p E(z)=\sqrt{-1}\int_{\mc{X}/\mc{T}}\p\langle \p u\wedge \b{\p}u\rangle,\quad \p \b{\p}E(z)=\sqrt{-1}\int_{\mc{X}/\mc{T}}\p\b{\p}\langle \p u\wedge \b{\p}u\rangle.
\end{align*}
The holomorphic cotangent bundle $T^*\mc{X}$ has the following decomposition:
\begin{align*}
T^*\mc{X}=\mc{H}^*\oplus \mc{V}^*,	
\end{align*}
where $\mc{H}^*$ and $\mc{V}^*$ are defined in (\ref{HV*}). By using the above decomposition,
the first and second variational formulas are obtained as  follows;
see  Subsection \ref{subsec1.3} for the definition of the connection
$\n$ and the notations.

\begin{thm}\label{Introthm4}
The first variation of energy is 
\begin{align*}
\frac{\p E(z)}{\p z^\alpha}&=\int_{\mc{X}/\mc{T}}\sqrt{-1}\langle \p^Vu\wedge \n_{\frac{\delta}{\delta z^\alpha}}\b{\p}^Vu\rangle\\
&=-\langle A_{\alpha},du\rangle,
\end{align*}
where $A_{\alpha}=A^v_{\alpha\b{v}}u^j_vd\b{v}\otimes \frac{\p}{\p x^j}\in A^1(\mc{X}_z, u^*TN)$,  $ A^v_{\alpha\b{v}}=\p_{\b{v}}(-\phi_{\alpha\b{v}}\phi^{v\b{v}})$. 
\end{thm}
\begin{thm}\label{Introthm5}
	The second variation of energy is 
	\begin{multline*}
	\frac{\p^2E(z)}{\p z^\alpha\p\b{z}^\beta}=	-2\int_{\mc{X}/\mc{T}}R\left(\frac{\p u}{\p v},\frac{\delta u}{\delta z^\alpha},\frac{\p u}{\p \b{v}},\frac{\delta u}{\delta \b{z}^\beta}\right)\sqrt{-1}\delta v\wedge \delta\b{v}\\
	+2\int_{\mc{X}/\mc{T}}\langle \n_{\frac{\delta}{\delta\b{z}^\beta}}\p^Vu\wedge\n_{\frac{\delta}{\delta z^\alpha}}\b{\p}^Vu\rangle  .
	\end{multline*}
\end{thm}
By using Cauchy-Schwarz inequality,  
for any $\xi=\xi^\alpha\frac{\p}{\p z^\alpha}\in T_z\mc{T}$, one has
	\begin{align*}
	\left|\xi^\alpha\frac{\p E(z)}{\p z^\alpha}\right|^2\leq E(z)\cdot \int_{\mc{X}/\mc{T}}\langle \n_{\b{\xi}^\beta\frac{\delta}{\delta\b{z}^\beta}}\p^Vu\wedge\n_{\xi^\alpha\frac{\delta}{\delta z^\alpha}}\b{\p}^Vu\rangle, 
	\end{align*}
see Lemma \ref{lemma5}. If $(N,g)$ has non-positive Hermitian sectional curvature, then
\begin{align*}
	\frac{\p^2 E(z)^{-1}}{\p z^\alpha\p\b{z}^\beta}\xi^\alpha\b{\xi}^\beta=-\frac{1}{E^2}\left(\frac{\p^2E(z)}{\p z^\alpha\p\b{z}^\beta}-\frac{2}{E}\frac{\p E(z)}{\p z^\alpha}\frac{\p E(z)}{\p\b{z}^\beta}\right)\xi^\alpha\b{\xi}^\beta\leq 0,
\end{align*}
which completes the proof of Theorem \ref{Introthm1}. By using the following two identities 
\begin{align*}
\sqrt{-1}\p\b{\p}\log E=-E\sqrt{-1}\p\b{\p}E^{-1}+E^{-2}\sqrt{-1}\p E\wedge \b{\p}E	
\end{align*}
and 
\begin{align*}
	\sqrt{-1}\p\b{\p}E =E\sqrt{-1}\p\b{\p}\log E+E^{-1}\sqrt{-1}\p E\wedge \b{\p}E,
	\end{align*} 
we obtain the plurisubharmonicity of $\log E(z)$ and $E(z)$.

As applications  we find that
the second variation of logarithmic  energy function is related to the Weil-Petersson metric. More precisely
\begin{thm}\label{Introthm6}
Let $(N,h)$ be a Hermitian manifold and fix a smooth map $u_0: \Sigma_g\to N$. If there is a unique harmonic map $u(z):(\mc{X}_z,\Phi_z)\to (N, g=\text{Re}\ h)$ in the homotopy class $[u_0]$ for each $z\in\mc{T}$, moreover $u(z_0)$ is  holomorphic (resp. anti-holomorphic) and totally geodesic
on $\mc{X}_{z_0}$, then 
	\begin{align*}
	\sqrt{-1}\p\b{\p}\log E(z)|_{z=z_0}=\frac{\omega_{WP}}{2\pi(g-1)}.	
	\end{align*}
\end{thm}
\begin{cor}[{\cite[Theorem 2.6]{FT}}]\label{Introcor7}
 If $u(z_0)=Id: (\mc{X}_{z_0}, \Phi_{z_0})\to (\mc{X}_{z_0}, \Phi_{z_0)}$ is identity, then  	
\begin{align*}
\sqrt{-1}\p\b{\p}E(z)|_{z=z_0}=2\omega_{WP}.	
\end{align*}
\end{cor}

We consider also the harmonic maps $u: M\to \mathcal X_z$ from a Riemannian manifold
$M $ to Riemann surfaces
 $(\mc{X}_z,\Phi_z)$
as in \cite{KWZ}, but with further assumption that
$M$  is a K\"ahler.
The energy function $E(z)$ is again defined on Teichm\"uller space
\cite{KWZ, Yamada}. We show that the variation is again related to the
Weil-Petersson metric.
\begin{thm}\label{Introthm8}
Let $(M,\omega_g)$ be a compact K\"ahler manifold and fix a smooth map $u_0: M\to \Sigma_g$, let $E(z)$ denote the energy function of harmonic maps from $(M,g)$ to $(\mc{X}_z,\Phi_z)$ in the class $[u_0]$, where $g$ is the  Riemannian metric associated to $\omega_g$.	If $u(z_0)$ is holomorphic or anti-holomorphic for some $z_0\in\mc{T}$, then 
\begin{align*}
\sqrt{-1}\p\b{\p}\log E(z)|_{z=z_0}=\frac{\omega_{WP}}{2\pi(g-1)}.	
\end{align*}
\end{thm}
As a corollary, we obtain
\begin{cor}\label{Introcor9}
If $M$ is a Riemann surface,  and   $u(z_0)$ is holomorphic or anti-holomorphic, then 
\begin{align*}
\sqrt{-1}\p\b{\p}E(z)|_{z=z_0}=	|\deg u(z_0)|\cdot 2\omega_{WP}.
\end{align*}
Here $\deg u(z_0)$ is the degree of $u(z_0)$.
\end{cor}
In particular, if $u(z_0)$ is the identity map, then 
$$ \sqrt{-1}\p\b{\p}E(z)|_{z=z_0}=2\omega_{WP},$$
which was proved by M. Wolf  \cite[Theorem 5.7]{Wolf1}.

\begin{rem}
	In many situations, the harmonic maps are $\pm$ holomorphic (i.e. holomorphic or anti-holomorphic) automatically. For example, 
	\begin{itemize}
	\item[(i)]	(Eells and Wood \cite{EW}) Let $X$ and $Y$ be compact Riemann surfaces and $f$ a harmonic map from $X$ to $Y$ with respect to some K\"ahler metrics. If $f$ satisfies the following condition then $f$ is $\pm$ holomorphic :
$$e(X)+|\deg f\cdot e(Y)|>0$$ 
where $e(X)$ and $e(Y)$ are the Euler numbers of $X$ and $Y$ respectively and $\deg(f)$ is the degree of the map $f:X\to Y$.
\item[(ii)] (Ono \cite{Ono}) If $(M^n,\omega)$ is a compact K\"ahler manifold with negative first Chern class and satisfies 
 \begin{align*}
 n|f^*c_1(N)\cdot c_1(M)^{n-1}[M]|>|c_1(M)^n[M]|,	
 \end{align*}
and $f$ is a harmonic from $M$ to  a compact hyperbolic Riemann $N$, then $f$ is $\pm$ holomorphic.
\item[(iii)] (Siu \cite{Siu}) Let $M$ and $N$ be compact K\"a{}hler manifolds and assume that $N$ has strongly negative curvature in the sense of Siu. Let $f$ be a harmonic map from $M$ to $N$ with respect to the K\"ahler metrics. If there is a point in $M$ where the rank of $df$ is greater than or equal to four, then $f$ is $\pm$ holomorphic.
\item[(iv)] (Siu and Yau \cite{SY}) Let $(M, h)$ be a compact K\"ahler manifold of dimension $n\geq 2$ with positive
holomorphic bisectional curvature. Then any energy minimizing map $f:\mb{P}^1\to M$ must be $\pm$ holomorphic.
	\end{itemize}
 \end{rem}
  
This article is organized as follows. In Section \ref{sec1}, we fix notations and recall
some basic facts on Teichm\"uller curve and harmonic maps.
In Section \ref{sec2}, we compute the first and the second variations
of the energy function (\ref{0.1}) and prove Theorem \ref{Introthm4},
\ref{Introthm5}. In Subsection \ref{subsec2.3} we  show the
plurisuperharmonicity of reciprocal energy and prove Theorem
\ref{Introthm1}, \ref{Introthm2} and Corollary \ref{Introcor3}. In the
last two sections, we study the relationship between the energy
function and the
Weil-Petersson metric, and prove Theorem \ref{Introthm6}, \ref{Introthm8} and Corollary \ref{Introcor7}, \ref{Introcor9}.

\section{Preliminaries}\label{sec1}
In this section, we shall fix the notations and recall some basic facts on Teichm\"uller curve and harmonic maps. The results in this section are well-known.
\subsection{Teichm\"uller curve}\label{subsec1.1}
Let $\mc{T}$ be Teichm\"uller space of a fixed surface of genus
$g\geq 2$. Let $\pi:\mc{X}\to \mc{T}$ be Teichm\"uller curve over Teichm\"uller space $\mc{T}$, namely the holomorphic family
of Riemann surfaces over $\mc{T}$,  the fiber $\mc{X}_z:=\pi^{-1}(z)$ being exactly the Riemann surface given by the complex structure $z\in\mc{T}$; see e.g. \cite[Section 5]{Ahlfors}. Denote by
$$(z;v)=(z^1,\cdots, z^m; v)$$
the local holomorphic coordinates of $\mc{X}$ with $\pi(z,v)=z$, where
$z=(z^1,\cdots z^m)$ denotes the local coordinates of $\mc{T}$ and $v$
denotes the local coordinates of Riemann surface $\mc{X}_z$,
$m=3g-3=\dim_{\mathbb C} \mc{T}$.  Let $K_{\mc{X}/\mc{T}}$ denote the
relative canonical line bundle over $\mc{X}$, i.e.,
$K_{\mc{X}/\mc{T}}|_{\mc{X}_z}
=K_{\mc{X}_z}$. The fibers $\mc{X}_z$ are equipped with hyperbolic metric 
\begin{align*}
\sqrt{-1}\phi_{v\b{v}}dv\wedge d\b{v}
\end{align*}
depending smoothly on the parameter $z$ and having negative constant curvature $-1$, namely, 
\begin{align}\label{1.1}
	\p_v\p_{\b{v}}\log\phi_{v\b{v}}=\phi_{v\b{v}},
\end{align}
where $\phi_{v\b{v}}:=\p_v\p_{\b{v}}\phi$.
From (\ref{1.1}), up to a scaling
function on $\mc{T}$  a metric (weight) $\phi$ on $K_{\mc{X}/\mc{T}}$ can
be chosen such that 
\begin{align}\label{1.2}
	e^{\phi}=\phi_{v\b{v}}.
\end{align}
For convenience,  we denote
$$\phi_{\alpha}:=\frac{\p \phi}{\p z^{\alpha}},\quad \phi_{\b{\beta}}:=\frac{\p \phi}{\p \b{z}^{\beta}},
\quad \phi_{v}:=\frac{\p \phi}{\p v},\quad \phi_{\b{v}}:=\frac{\p \phi}{\p \b{v}},$$
where $1\leq \alpha,\beta\leq m$.
 With respect to the $(1,1)$-form $\sqrt{-1}\p\b{\p}\phi$,
we have a canonical horizontal-vertical decomposition of $T\mc{X}$, $T\mc{X}=\mc{H}\oplus \mc{V}$, where
\begin{align}\label{HV}
\mc{H}=\text{Span}\left\{\frac{\delta}{\delta z^{\alpha}}=\frac{\p}{\p z^{\alpha}}+a^v_\alpha \frac{\p}{\p v}, 1\leq \alpha\leq m\right\},\quad \mc{V}=\text{Span}\left\{\frac{\p}{\p v}\right\},
\end{align}
where \begin{align}\label{1.3}
a^v_{\alpha}=-\phi_{\alpha\b{v}}\phi^{v\b{v}},
\end{align}
and $\phi^{v\b{v}}=(\phi_{v\b{v}})^{-1}$.
By duality,  $T^*\mc{X}=\mc{H}^*\oplus \mc{V}^*$, where 
\begin{align}\label{HV*}
\mc{H}^*=\text{Span}\left\{dz^{\alpha}, 1\leq \alpha\leq m\right\},\quad \mc{V}^*=\text{Span}\left\{\delta v=dv-a^v_{\alpha}dz^{\alpha}\right\}.
\end{align}
Moreover, the differential operators
\begin{align*}
\p^V=\frac{\p}{\p v}\otimes \delta v,\quad \p^H=\frac{\delta}{\delta z^{\alpha}}\otimes dz^{\alpha}, \quad \b{\p}^V=\frac{\p}{\p\bar v}\otimes \delta \bar v, \quad \b{\p}^H=\frac{\delta}{\delta \bar z^{\alpha}}\otimes d\bar z^{\alpha}
\end{align*}
are well-defined and satisfy 
\begin{align*}
d=\p+\b{\p},\quad \p=\p^H+\p^V,\quad \b{\p}=\b{\p}^V+\b{\p}^H	
\end{align*}
when acting on smooth functions of $\mc{X}$.
 The following two lemmas can be proved by direct computations.
\begin{lemma}[{\cite[Lemma 1.1]{Wan1}}]\label{lemma0}
  The $(1, 1)$-form $\sqrt{-1}\p\b{\p}\phi$ on $\mathcal X$
  has  the following horizontal-vertical decomposition:
\begin{align*}
  \sqrt{-1}\p\b{\p}\phi
  =c(\phi)+\sqrt{-1}\phi_{v\b{v}}\delta v\wedge \delta\b{v},	
\end{align*}
where $c(\phi)=\sqrt{-1}c(\phi)_{\alpha\b{\beta}}dz^{\alpha}\wedge d\b{z}^{\beta}$, $c(\phi)_{\alpha\b{\beta}}=\phi_{\alpha\b{\beta}}-\phi^{v\b{v}}\phi_{\alpha\b{v}}\phi_{v\b{\beta}}$. 
\end{lemma}
\begin{lemma}\label{lemma2}
For any smooth function $f$ on $\mc{X}$ we have 
\begin{multline*}
\p\b{\p}f=(f_{\alpha\b{\beta}}+f_{\alpha\b{v}}\o{a^v_\beta}+f_{v\b{\beta}}a^v_{\alpha}+f_{v\b{v}}a^v_{\alpha}\o{a^v_{\beta}})dz^{\alpha}\wedge d\b{z}^{\beta}+f_{v\b{v}}\delta v\wedge \delta \b{v}	\\
+\frac{\delta}{\delta z^{\alpha}}\left(\frac{\p f}{\p \b{v}}\right)dz^{\alpha}\wedge \delta\b{v}+\frac{\delta}{\delta\b{z}^{\beta}}\left(\frac{\p f}{\p v}\right)\delta v\wedge d\b{z}^{\beta}.
\end{multline*}
	
\end{lemma}

Consider the following tensor
 \begin{align}\label{A-pre}
\b{\p}^V\frac{\delta}{\delta z^{\alpha}}=(\p_{\b{v}}a^v_{\alpha}) \frac{\p}{\p v}\otimes \delta\b{v}\in A^0(\mc{X}, \text{End}(\mc{V})).
\end{align}
We denote its component and its dual 
with respect to the metric 
$\sqrt{-1}\phi_{v\b{v}}\delta v\wedge \delta \b{v}$ as  
\begin{align}\label{1.4}
A^v_{\alpha\b{v}}=\p_{\b{v}}a^v_\alpha=\p_{\b{v}}(-\phi^{v\b{v}}\phi_{\alpha\b{v}}),\quad A_{\alpha\b{v}\b{v}}=A^v_{\alpha\b{v}}\phi_{v\b{v}}.	 
\end{align}
Lemma \ref{lemma1} (i) below shows
that its restriction to each fiber
is a harmonic element
representing  the Kodaira-Spencer class $\rho(\frac{\p}{\p z^{\alpha}})$, where
$$\rho: T_z\mc{T}\to H^1(\mc{X}_z,T_{\mc{X}_z})$$ is
the Kodaira-Spencer map.

\begin{lemma}\label{lemma1}
\cite[Proposition 2, 3]{Sch}
The following identities hold:
\begin{enumerate}
  \item[(i)] $\p_vA_{\alpha\b{v}\b{v}}=0$;
  \item[(ii)] $(\Box+1)c(\phi)_{\alpha\b{\beta}}=A^{v}_{\alpha\b{v}}A^{\b{v}}_{\b{\beta}v}$ where $\Box=-\phi^{v\b{v}}\p_v\p_{\b{v}}$. 
\end{enumerate}
\end{lemma}
\begin{defn}
The Weil-Petersson metric $\omega_{WP}$ on Teichm\"uller space $\mc{T}$ is defined by 
\begin{align}\label{1.19}
\omega_{WP}=\sqrt{-1}G_{\alpha\b{\beta}}dz^{\alpha}\wedge d\b{z}^{\beta},\quad G_{\alpha\b{\beta}}(z)=\int_{\mc{X}_z}A^v_{\alpha\b{v}}\o{A^v_{\beta\b{v}}}\sqrt{-1}\phi_{v\b{v}}dv\wedge d\b{v}.	
\end{align}
\end{defn}
By Lemma \ref{lemma1} (ii) and Stokes' theorem, the Weil-Petersson
metric  can
also be expressed as 
\begin{align}\label{1.29}
	G_{\alpha\b{\beta}}(z)=\int_{\mc{X}_z}c(\phi)_{\alpha\b{\beta}}\sqrt{-1}\phi_{v\b{v}}dv\wedge d\b{v}.
\end{align}

\subsection{Harmonic maps
  from Riemann surfaces to
a Riemannian manifold}\label{subsec1.3}
 Let 
\begin{align*}
\Phi_z=\phi_{v\b{v}}(dv\otimes d\b{v}+d\b{v}\otimes dv)	
\end{align*}
denote the Riemannian metric on $\mc{X}_z$ associated to the fundamental $(1,1)$-form $$\sqrt{-1}\phi_{v\b{v}}dv\wedge d\b{v}=\sqrt{-1}\phi_{v\b{v}}(dv\otimes d\b{v}-d\b{v}\otimes dv).$$
Let $T\mc{X}_z$ be the holomorphic tangent bundle of $\mc{X}_z$ and $T_{\mb{C}}\mc{X}_z=T\mc{X}_z\oplus\o{T\mc{X}_z}$  the complex tangent bundle.
For any smooth map $u:  (\mc{X}_z,\Phi_z)\to (N,g)$
the differential
 $du$ is a section of the
bundle $T^*_{\mb{C}}\mc{X}_z\otimes u^*TN$. 
 Let $\{x^i\}_{1\leq i\leq \dim N}$ denote a local coordinate system of $N$
and $v$ the local complex coordinate on $\mc{X}_z$. Then
$du\in A^0(\mc{X}_z, T^*_{\mb{C}}\mc{X}_z\otimes u^*TN)$ is locally expressed as 
 \begin{align*}
 du=\frac{\p u^i}{\p v}dv\otimes \frac{\p}{\p x^i}+ \frac{\p u^i}{\p \b{v}}d\b{v}\otimes \frac{\p}{\p x^i}.	
 \end{align*}
 The energy density
and the energy are  defined by
\begin{align*}
	|du|^2:=(du,du)=2g_{ij}u^i_v u^j_{\b{v}}\phi^{v\b{v}},
\end{align*}
\begin{aligns}\label{energy}
E(u)&:=\frac{1}{2}\|du\|^2:=\frac{1}{2}\int_{\mc{X}_z} |du|^2d\mu_{\Phi_z}\\
&=\int_{\mc{X}_z} (g_{ij}u^i_v u^j_{\b{v}}\phi^{v\b{v}})\sqrt{-1}\phi_{v\b{v}}dv\wedge d\b{v}\\
&=\int_{\mc{X}_z} g_{ij}u^i_v u^j_{\b{v}}\sqrt{-1}dv\wedge d\b{v},
\end{aligns}
where $u_v^i:=\frac{\p u^i}{\p v}$
and
$d\mu_{\Phi_z}=\sqrt{-1}\phi_{v\b{v}}dv\wedge d\b{v}$ is the Riemannian volume form of $\Phi_z$.
The harmonic equation  is 
\begin{aligns}\label{harmonic}
\p_{\b{v}}u^i_v+\Gamma^i_{jk}u^j_v u^k_{\b{v}}=0;	
\end{aligns}
see e.g. \cite[(1.2.10)]{Xin}. Here
$$\Gamma^{k}_{ij}=\frac{1}{2}g^{kl}\left(\p_jg_{il}+ \p_i g_{jl}- \p_lg_{ij}\right)$$
denotes the Christoffel symbols on $(N, g)$.

Let $\{u(z)\}_{z\in\mc{T}}$ be a smooth family of harmonic maps
$u(z):(\mc{X}_z,\Phi_z)\to (N,g)$, $z\in \mc{T}$. We shall treat
it as  a smooth map $u$,
\begin{align*}
u: \mc{X}\to N,\quad (z,v)\mapsto u(z,v):=(u(z))(v).	
\end{align*}
 Note that $\mc{V}^*$ is a holomorphic line bundle over $\mc{X}$ with
 holomorphic frame $\{\delta v\}$, which is equipped with a Hermitian
 metric $(\phi_{v\b{v}}^{-1}=e^{-\phi})$, thus there is a natural
 induced connection $\n$ on $\mc{V}^*\otimes u^*TN$ from the Chern
 connection of $\mc{V}^*$ and the
 Levi-Civita connection of $TN$, i.e. for $X\in T_{\mathbb C}X$,
\begin{align*}
\n_X: 	A^0(\mc{X}, \mc{V}^*\otimes u^*TN)\to A^1(\mc{X}, \mc{V}^*\otimes u^*TN),
\end{align*}
By conjugation, we obtain a connection $\n$ on $\o{\mc{V}}^*\otimes u^*TN$.
More precisely for any $f=f^i_v\delta v\otimes \frac{\p}{\p x^i}+f^i_{\b{v}}\delta\b{v}\otimes\frac{\p}{\p x^i}\in A^0(\mc{X}, (\mc{V}^*\oplus\o{\mc{V}}^*)\otimes u^*TN)$ and any vector $X\in T_{\mb{C}}\mc{X}$
\begin{align*}
\n_Xf=(\n_X f^i_v)\delta v\otimes \frac{\p}{\p x^i}	+(\n_X f^i_{\b{v}})\delta\b{v}\otimes \frac{\p}{\p x^i},
\end{align*}
where
\begin{align*}
\n_X f^i_v:=X(f^i_v)+\Gamma^i_{kl}f_v^k X(u^l)-(\p\phi)(X)f^i_v	
\end{align*}
and 
\begin{align*}
\n_X f^i_{\b{v}}:=X(f^i_{\b{v}})+\Gamma^i_{kl}f^k_{\b{v}} X(u^l)-(\b{\p}\phi)(X)f^i_{\b{v}}.	
\end{align*}
Denote $\n_v:=\n_{\frac{\p}{\p  
    v}}$,$\n_{\b{v}}:=\n_{\frac{\p}{\p \b{v}}}$ for notational convinience.
In particular for $u^i_v\delta v\otimes \frac{\p}{\p x^i}\in A^0(\mc{X},
\mc{V}^*\otimes u^*TN)$ we have
\begin{align*}
\n_{\b{v}}(u^i_v\delta v\otimes \frac{\p}{\p x^i})=	(\n_{\b{v}}u^i_v)\delta v\otimes\frac{\p}{\p x^i},\quad \n_{\b{v}}u^i_v:=\p_{\b{v}}u^i_v+\Gamma^i_{jk}u^j_v u^k_{\b{v}}.
\end{align*}
By (\ref{harmonic}), $u$ is a harmonic map if and only if
\begin{align}\label{harmonic1}
\n_{\b{v}}u^i_v=0.	
\end{align}

\section{Variations of energy on Teichm\"uller space}\label{sec2}

In this section we will calculate the first and the second variations of 
the energy $E(u(z))$ for harmonic maps $u(z):\mc{X}_z\to N$.
 Fix a smooth map $u_0:\Sigma\to N$ from a surface $\Sigma$ of genus $g$ to $N$. We assume that $u:\mc{X}\to N$ is a smooth map such that $u(z):\mc{X}_z\to N$ is a harmonic map in the homotopy class $[u_0]$. Then the following 
 function
\begin{align*}
E(z):=E(u(z))	
\end{align*}
is smooth 
on Teichm\"uller space $\mc{T}$. 

\subsection{The first variation}\label{subsec2.1}
Consider a smooth map 
\begin{align*}
u:\mc{X}\to N,\quad (z,v)\mapsto u(z,v).
\end{align*}
Recall that
\begin{align*}
du=\p u+\b{\p}u\in A^1(\mc{X}, u^*TN),	
\end{align*}
with 
\begin{align*}
\p u:=\p u^i\otimes \frac{\p}{\p x^i}=(u^i_z dz+u^i_vdv)\otimes \frac{\p}{\p x^i}\in A^{1,0}(\mc{X}, u^*TN)	
\end{align*}
 the $(1,0)$-component of $du$, and $\b{\p}u=\o{\p u}$  the $(0,1)$-component of $du$. Let 
\begin{align}\label{1.6}
	\langle \p u\wedge \b{\p}u\rangle=g_{ij}(u(z,v)) \p u^i\wedge \b{\p} u^j\in A^{1,1}(\mc{X}) 
\end{align}
 denote the two-form on $\mc{X}$ obtained by combining the wedge product in $\mc{X}$ with the Riemannian metric $\langle ,\rangle$ on $u^*TN$. 
The corresponding 
the energy $E(z)$ function (\ref{energy}) can be written as 
\begin{align}\label{2.1}
E(z)=\sqrt{-1}\int_{\mc{X}/\mc{T}}\langle \p u\wedge \b{\p}u\rangle. 	
\end{align}
Here  we view
$\int_{\mc{X}/\mc{T}}$ as
\begin{align*}
	\int_{\mc{X}/\mc{T}}: A^{2+k}(\mc{X})\to A^k(\mc{T})
\end{align*}
 denotes the integral along fibers (see e.g. \cite[Section 2.1]{Sch}),
 and $\p$, $\b{\p}$-operators commute with $\int_{\mc{X}/\mc{T}}$. 

The variations of $E(z)$ are
\begin{align}\label{1.5}
\p E(z)=\sqrt{-1}\int_{\mc{X}/\mc{T}}\p\langle \p u\wedge \b{\p}u\rangle,\quad \p \b{\p}E(z)=\sqrt{-1}\int_{\mc{X}/\mc{T}}\p\b{\p}\langle \p u\wedge \b{\p}u\rangle.
\end{align}
Note that $\p\langle \p u\wedge \b{\p}u\rangle\in A^3(\mc{X})$, which
can be decomposed in terms of the frame $\wedge^3\{dz^{\alpha},
d\b{z}^{\beta},\delta v, \delta\b{v}\}$, we denote by $\left[\p\langle
  \p u\wedge \b{\p}u\rangle\right]^{(\delta v\wedge \delta \b{v})}$
the component of $\p\langle \p u\wedge \b{\p}u\rangle$
containing $\delta
v\wedge\delta\b{v}$. We shall need the following
\begin{lemma}\label{lemma3}
The  $\delta
v\wedge\delta\b{v}$-component is
\begin{align*}
	\left[\p\langle \p u\wedge \b{\p}u\rangle\right]^{(\delta v\wedge \delta \b{v})}=\langle \p^Vu\wedge \n_{\frac{\delta}{\delta z^\alpha}}\b{\p}^Vu\rangle\wedge dz^\alpha,	
	\end{align*}
	where 
	\begin{align*}
	\p^V u:=u^i_v\delta v\otimes\frac{\p}{\p x^i}\in A^0(\mc{X},\mc{V}^*\otimes u^*TN)\subset A^{1,0}(\mc{X}, u^*TN)	
	\end{align*}
and 
\begin{align*}
	\b{\p}^Vu:= u^i_{\b{v}}\delta\b{v}\otimes \frac{\p}{\p x^i} \in A^0(\mc{X},\o{\mc{V}}^*\otimes u^*TN)\subset A^{0,1}(\mc{X}, u^*TN).
\end{align*}
\end{lemma}
\begin{proof}
	For any fixed point $(z_0,v_0)\in \mc{X}$, we choose a normal coordinate system $\{x^i\}$ around $u(z_0,v_0)$ such that 
\begin{align}\label{normal}
g_{ij}(u(z_0,v_0))=\delta_{ij},\quad (dg_{ij})(u(z_0,v_0))=0.
\end{align}
From (\ref{harmonic}), one has $\p_{\b{v}}u^i_{v}(z_0,v_0)=0$.
  Lemma \ref{lemma2} implies that
\begin{multline}\label{1.7}
	\p\b{\p}u^i=(u_{\alpha\b{\beta}}+u_{\alpha\b{v}}\o{a^v_\beta}+u_{v\b{\beta}}a^v_{\alpha})dz^\alpha\wedge d\b{z}^\beta\\+\frac{\delta}{\delta z^\alpha}u^i_{\b{v}}dz^\alpha\wedge \delta\b{v}+\frac{\delta}{\delta\b{z}^\beta}u^i_v\delta v\wedge d\b{z}^\beta.
\end{multline}
	Thus at $(z_0,v_0)\in \mc{X}$, one has
	\begin{align*}
	\left[\p\langle \p u\wedge \b{\p}u\rangle\right]^{(\delta v\wedge \delta \b{v})} &=	\left[\p g_{ij}\wedge \p u^i\wedge \b{\p}u^j-g_{ij}\p u^i\wedge \p\b{\p}u^j\right]^{(\delta v\wedge \delta \b{v})}\\
	&=\left[-\p u^i\wedge \p\b{\p}u^i\right]^{(\delta v\wedge \delta \b{v})}\\
	&=\left[-\p u^i\wedge (\frac{\delta}{\delta z^\alpha}u^i_{\b{v}})dz^\alpha\wedge \delta\b{v}\right]^{(\delta v\wedge \delta \b{v})}\\
	&=g_{ij}u^i_{v} \o{\n_{\frac{\delta}{\delta \b{z}^\alpha}}u^j_{v}}\delta v\wedge \delta \b{v}\wedge dz^\alpha\\
	&=\langle \p^Vu\wedge \n_{\frac{\delta}{\delta z^\alpha}}\b{\p}^Vu\rangle\wedge dz^\alpha.
	\end{align*}
	Since both $\left[\p\langle \p u\wedge \b{\p}u\rangle\right]^{(\delta v\wedge \delta \b{v})}$ and $\langle \p^Vu\wedge \n_{\frac{\delta}{\delta z^\alpha}}\b{\p}^Vu\rangle\wedge dz^\alpha$ are globally defined,  independent of the normal coordinate system $\{x^i\}$, and the point $(z_0,v_0)$ is arbitrary, 
	\begin{align*}
		\left[\p\langle \p u\wedge \b{\p}u\rangle\right]^{(\delta v\wedge \delta \b{v})}=\langle \p^Vu\wedge \n_{\frac{\delta}{\delta z^\alpha}}\b{\p}^Vu\rangle\wedge dz^\alpha
	\end{align*}
on $\mc{X}$.
\end{proof}
\begin{thm}\label{thm2}
The first variation of energy is 
\begin{align*}
\frac{\p E(z)}{\p z^\alpha}=\int_{\mc{X}/\mc{T}}\sqrt{-1}\langle \p^Vu\wedge \n_{\frac{\delta}{\delta z^\alpha}}\b{\p}^Vu\rangle.
\end{align*}
\end{thm}
	\begin{proof}
	By (\ref{1.5}) and 	Lemma \ref{lemma3}, 
	\begin{align*}
		\p E(z)&=\int_{\mc{X}/\mc{T}}\sqrt{-1}\p\langle \p u\wedge \b{\p}u\rangle\\
		&=\int_{\mc{X}/\mc{T}}\sqrt{-1}\left[\p\langle \p u\wedge \b{\p}u\rangle\right]^{(\delta v\wedge \delta \b{v})}\\
		&=\left(\int_{\mc{X}/\mc{T}}\sqrt{-1}\langle \p^Vu\wedge \n_{\frac{\delta}{\delta z^\alpha}}\b{\p}^Vu\rangle\right) dz^\alpha,
	\end{align*}
	which completes the proof.
	\end{proof}
Now we will give another formula on the first variation of energy function. Denote 
\begin{align}\label{A-a-1}
A_{\alpha}=A^v_{\alpha\b{v}}u^j_vd\b{v}\otimes \frac{\p}{\p x^j}\in A^1(\mc{X}_z, u^*TN),\quad A^v_{\alpha\b{v}}=\p_{\b{v}}(-\phi_{\alpha\b{v}}\phi^{v\b{v}}).
\end{align}
Then
\begin{align*}
\langle A_{\alpha}, du\rangle & =	\langle A^v_{\alpha\b{v}}u^j_vd\b{v}\otimes \frac{\p}{\p x^j}, u^i_v dv\otimes \frac{\p}{\p x^i}+u^i_{\b{v}}d\b{v}\otimes \frac{\p}{\p x^i}\rangle\\
&=\int_{\mc{X}_z}g_{ij}u^i_vu^j_vA^v_{\alpha\b{v}}\sqrt{-1}dv\wedge d\b{v}.
\end{align*}
\begin{thm}\label{thm1}
	The first variation of energy function is 
\begin{align*}
	\frac{\p E(z)}{\p z^\alpha}=-\langle A_{\alpha},du\rangle .
\end{align*}
\end{thm}
\begin{proof}
	From Theorem \ref{thm2} we find
	\begin{align*}
	\frac{\p E(z)}{\p z^\alpha} &=	\int_{\mc{X}/\mc{T}}\sqrt{-1}\langle \p^Vu\wedge \n_{\frac{\delta}{\delta z^\alpha}}\b{\p}^Vu\rangle\\
	&=\int_{\mc{X}_z}g_{ij}u^i_{v} \n_{\frac{\delta}{\delta z^\alpha}}u^j_{\b{v}}\sqrt{-1}dv\wedge d\b{v}\\
	&=\int_{\mc{X}_z}g_{ij}u^i_v\left(\n_{\b{v}}\frac{\delta u^j}{\delta z^\alpha}-A^v_{\alpha\b{v}}u^j_v\right)\sqrt{-1}dv\wedge d\b{v}\\
		&=-\int_{\mc{X}_z} g_{ij}u^i_v u^j_v A_{\alpha\b{v}}^v \sqrt{-1}dv\wedge d\b{v}\\
	&=-\langle A_\alpha,du\rangle,
	\end{align*}
where the fourth equality follows from Stokes' theorem and the harmonic equation (\ref{harmonic1}), and the third equality holds by 
\begin{aligns}\label{1.14}
	\n_{\frac{\delta}{\delta z^\alpha}}u^j_{\b{v}}&=\frac{\delta}{\delta z^\alpha}u^j_{\b{v}}+\Gamma^j_{kl}u^l_{\b{v}}\frac{\delta u^k}{\delta z^\alpha}\\
	&=\frac{\p}{\p\b{v}}(\phi_{\alpha\b{v}}\phi^{\b{v}v})u^j_v+\frac{\p}{\p\b{v}}\left(\frac{\delta u^j}{\delta z^\alpha}\right)+\Gamma^j_{kl}u^l_{\b{v}}\frac{\delta u^k}{\delta z^\alpha}\\
	&=-A^v_{\alpha\b{v}}u^j_v+\n_{\b{v}}\frac{\delta u^j}{\delta z^\alpha}.
\end{aligns}
\end{proof}

\subsection{The second variation}\label{subsec2.2}

We first recall the definition of Hermitian sectional curvature on a
Riemannian manifold $(N, g)$. Let $\n^N$ be the Levi-Civita connection
of Riemannian manifold $(N,g)$.
Recall that the Riemann curvature endomorphism $R\in
A^2(N,\text{End}(TN))$ is
\begin{align*}
R(X,Y)Z=\n^N_X\n^N_YZ-\n^N_Y\n^N_XZ-\n^N_{[X,Y]}Z.
\end{align*} Recall also the notation
\begin{align*}
R(X,Y,Z,W)=-\langle R(X,Y)Z,W\rangle, 	
\end{align*}
and
\begin{align*}
R_{ikjl}:=R\left(\frac{\p}{\p x^i}, \frac{\p}{\p x^k},\frac{\p}{\p x^j},\frac{\p}{\p x^l}\right).	
\end{align*}
By a direct calculation, one has
\begin{multline*}
R_{ikjl}=-\frac{1}{2}\left(\frac{\p^2 g_{ij}}{\p x^k\p x^l}+\frac{\p^2 g_{kl}}{\p x^i\p x^j}-\frac{\p^2 g_{il}}{\p x^k\p x^j}-\frac{\p^2 g_{kj}}{\p x^i\p x^l}\right)\\-g^{mn}\left(\Gamma^m_{ij}\Gamma^n_{kl}-\Gamma^m_{il}\Gamma^n_{kj}\right).
\end{multline*}
The sectional curvature is defined by 
\begin{align*}
K(X\wedge Y)=\frac{R(X,Y,X,Y)}{\|X\|^2\|Y\|^2-\langle X,Y\rangle^2}.	
\end{align*}
The Riemann curvature tensor $R$ can be
extended  on the complexified tangent bundle $TN\otimes\mb{C}$.
We recall the following  curvature condition
of Siu \cite{Siu}
and Sampson \cite{Sampson}.

\begin{defn}[\cite{Siu, Sampson, Toledo}]\label{HSC}
 For any $X,Y\in TN\otimes \mb{C}$, the {\it Hermitian sectional curvature} on the plane $X\wedge Y$ is defined by 
\begin{align*}
K_{\mb{C}}(X\wedge Y):=\frac{R(X,Y,\o{X},\o{Y})}{\|X\|^2\|Y\|^2-|\langle X,\o{Y}\rangle|^2}.	
\end{align*}
The Riemannian manifold $(N,g)$ is said to have {\it non-positive (resp. strictly negative) Hermitian sectional curvature} if 
\begin{align*}
K_{\mb{C}}(X\wedge Y)\leq 0	\quad (resp. <0)
\end{align*}
	for any $X, Y\in TN\otimes \mb{C}$ with $X\wedge Y\neq 0$.
\end{defn}

Recall the notation $\left[\p\b{\p}\langle \p u\wedge
  \b{\p}u\rangle\right]^{(\delta v\wedge \delta \b{v})}$, 
 the part of $\p\b{\p}\langle \p u\wedge \b{\p}u\rangle$ containing $\delta v\wedge \delta \b{v}$.
Then 
\begin{lemma}\label{lemma4} It holds
\begin{multline*}
\left[\p\b{\p}\langle \p u\wedge \b{\p}u\rangle\right]^{(\delta v\wedge \delta \b{v})}=-2R\left(\frac{\p u}{\p v},\frac{\delta u}{\delta z^\alpha},\frac{\p u}{\p \b{v}},\frac{\delta u}{\delta \b{z}^\beta}\right)\delta v\wedge \delta\b{v}\wedge dz^\alpha\wedge d\b{z}^\beta\\+2\langle \n_{\frac{\delta}{\delta\b{z}^\beta}}\p^Vu\wedge\n_{\frac{\delta}{\delta z^\alpha}}\b{\p}^Vu\rangle\wedge dz^\alpha\wedge d\b{z}^\beta,
\end{multline*}	
where 
\begin{align*}
\frac{\p u}{\p v}=u^i_v\frac{\p}{\p x^i},\quad 	\frac{\delta u}{\delta z^\alpha}=\frac{\delta u^i}{\delta z^{\alpha}}\frac{\p}{\p x^i},\quad \frac{\p u}{\p\b{v}}=u^i_{\b{v}}\frac{\p}{\p x^i},\quad 	\frac{\delta u}{\delta \b{z}^\beta}=\frac{\delta u^i}{\delta \b{z}^{\beta}}\frac{\p}{\p x^i}.
\end{align*}
\end{lemma}
\begin{proof}
	From (\ref{1.6}), one has
	\begin{multline*}
\p\b{\p}\langle \p u\wedge \b{\p}u\rangle =	(\p\b{\p} g_{ij}\wedge \p u^i-\p g_{ij}\wedge \p\b{\p} u^i)\wedge \b{\p}u^j\\+(\b{\p} g_{ij}\wedge \p u^i-g_{ij}\p\b{\p}u^i)	\wedge \p\b{\p}u^j.	
	\end{multline*}
By taking a normal coordinates system $\{x^i\}$ around $u(z_0,v_0)$
for any fixed point $(z_0,v_0)\in\mc{X}$ as in (\ref{normal}), we get that, at the point $(z_0,v_0)$,
\begin{align*}
	\p\b{\p}\langle \p u\wedge \b{\p}u\rangle=\p\b{\p}g_{ij}\wedge \p u^i\wedge \b{\p} u^j-(\p\b{\p}u^i)^2.
\end{align*}
By (\ref{1.7}) we have further
\begin{align*}
&\quad \left[\p\b{\p}\langle \p u\wedge \b{\p}u\rangle\right]^{(\delta v\wedge \delta \b{v})}\\
&=\left[\p\b{\p}g_{ij}\wedge \p u^i\wedge \b{\p} u^j-(\p\b{\p}u^i)^2\right]^{(\delta v\wedge \delta \b{v})}\\
&=\left[(\p_k\p_{l}g_{ij})\p u^k\wedge \b{\p}u^l\wedge \p u^i\wedge \b{\p} u^j -(\frac{\delta}{\delta z^\alpha}u^i_{\b{v}}dz^\alpha\wedge \delta\b{v}+\frac{\delta}{\delta\b{z}^\beta}u^i_v\delta v\wedge d\b{z}^\beta)^2\right]^{(\delta v\wedge \delta \b{v})}\\
&=\left((\p_k\p_{l}g_{ij})(\frac{\delta u^k}{\delta z^\alpha}\frac{\delta u^l}{\delta\b{z}^\beta}u^i_v u^j_{\b{v}}-\frac{\delta u^k}{\delta z^\alpha}\frac{\delta u^j}{\delta\b{z}^\beta}u^i_v u^l_{\b{v}}\right.\\
&\quad\left.-\frac{\delta u^i}{\delta z^\alpha}\frac{\delta u^l}{\delta\b{z}^\beta}u^k_v u^j_{\b{v}}+\frac{\delta u^i}{\delta z^\alpha}\frac{\delta u^j}{\delta\b{z}^\beta}u^k_v u^l_{\b{v}})
+2(\frac{\delta }{\delta z^\alpha}u^i_{\b{v}})(\frac{\delta}{\delta\b{z}^\beta}u^i_v)\right)
dz^\alpha\wedge d\b{z}^\beta\wedge \delta v\wedge \delta\b{v}\\
&=2\left(-R_{ikjl}\frac{\delta u^k}{\delta z^\alpha}\frac{\delta u^l}{\delta\b{z}^\beta}u^i_v u^j_{\b{v}}+g_{ij}\n_{\frac{\delta}{\delta \b{z}^\beta}}u^j_{v}\o{\n_{\frac{\delta}{\delta \b{z}^\alpha}}u^i_{v}}\right)\delta v\wedge \delta\b{v}\wedge dz^\alpha\wedge d\b{z}^\beta \\
&=-2R\left(\frac{\p u}{\p v},\frac{\delta u}{\delta z^\alpha},\frac{\p u}{\p \b{v}},\frac{\delta u}{\delta \b{z}^\beta}\right)\delta v\wedge \delta\b{v}\wedge dz^\alpha\wedge d\b{z}^\beta\\
&\quad+2\langle \n_{\frac{\delta}{\delta\b{z}^\beta}}\p^Vu\wedge\n_{\frac{\delta}{\delta z^\alpha}}\b{\p}^Vu\rangle\wedge dz^\alpha\wedge d\b{z}^\beta,
\end{align*}
where the second equality follows from (\ref{1.7}) and note that $[\p\b{\p}u^i]^{(\delta v\wedge \delta\b{v})}=0$ at the point $(z_0,v_0)$, 
 the fourth equality holds since 
$$\n_{\frac{\delta}{\delta\b{z}^\beta} }u^j_v=\frac{\delta}{\delta\b{z}^\beta}u^j_v+\Gamma^j_{kl}u^k_v\frac{\delta u^l}{\delta\b{z}^\beta}=\frac{\delta}{\delta\b{z}^\beta}u^j_v$$ 
at the point $(z_0,v_0)$ and
\begin{align*}
R_{ikjl}=-\frac{1}{2}\left(\frac{\p^2 g_{ij}}{\p x^k\p x^l}+\frac{\p^2 g_{kl}}{\p x^i\p x^j}-\frac{\p^2 g_{il}}{\p x^k\p x^j}-\frac{\p^2 g_{kj}}{\p x^i\p x^l}\right)
\end{align*}
at the point $u(z_0,v_0)$.
Since the point $(z_0,v_0)$ is arbitrary, we complete the proof.
\end{proof}
\begin{thm}\label{thm3}
	The second variation of the energy is 
	\begin{multline*}
	\frac{\p^2E(z)}{\p z^\alpha\p\b{z}^\beta}=	2\int_{\mc{X}/\mc{T}}-R\left(\frac{\p u}{\p v},\frac{\delta u}{\delta z^\alpha},\frac{\p u}{\p \b{v}},\frac{\delta u}{\delta \b{z}^\beta}\right)\sqrt{-1}\delta v\wedge \delta\b{v}\\
	+2\int_{\mc{X}/\mc{T}}\langle \n_{\frac{\delta}{\delta\b{z}^\beta}}\p^Vu\wedge\n_{\frac{\delta}{\delta z^\alpha}}\b{\p}^Vu\rangle  .
	\end{multline*}
\end{thm}
\begin{proof}
	By (\ref{1.5}) and Lemma \ref{lemma4} we have
	\begin{align*}
	\p\b{\p}E(z) &=	\sqrt{-1}\int_{\mc{X}/\mc{T}}\p\b{\p}\langle \p u\wedge \b{\p}u\rangle\\
	&=\sqrt{-1}\int_{\mc{X}/\mc{T}} \left[\p\b{\p}\langle \p u\wedge \b{\p}u\rangle\right]^{(\delta v\wedge \delta \b{v})}\\
	&=2\int_{\mc{X}/\mc{T}}\left(-R\left(\frac{\p u}{\p v},\frac{\delta u}{\delta z^\alpha},\frac{\p u}{\p \b{v}},\frac{\delta u}{\delta \b{z}^\beta}\right)\delta v\wedge \delta\b{v}\wedge dz^\alpha\wedge d\b{z}^\beta\right.\\
&\quad+\left.\langle \n_{\frac{\delta}{\delta\b{z}^\beta}}\p^Vu\wedge\n_{\frac{\delta}{\delta z^\alpha}}\b{\p}^Vu\rangle\wedge dz^\alpha\wedge d\b{z}^\beta\right)\\
&=2\int_{\mc{X}/\mc{T}}\left(-R\left(\frac{\p u}{\p v},\frac{\delta u}{\delta z^\alpha},\frac{\p u}{\p \b{v}},\frac{\delta u}{\delta \b{z}^\beta}\right)\sqrt{-1}\delta v\wedge \delta\b{v}\right.\\
&\quad\left.+\langle \n_{\frac{\delta}{\delta\b{z}^\beta}}\p^Vu\wedge\n_{\frac{\delta}{\delta z^\alpha}}\b{\p}^Vu\rangle\right) \cdot dz^\alpha\wedge d\b{z}^\beta,
	\end{align*}
 which completes the proof.
\end{proof}

\subsection{Plurisuperharmonicity}\label{subsec2.3}

In this subsection, we will prove the strict plurisubharmonicity of logarithmic energy function $\log E(z)$ and plurisuperharmoncity of reciprocal energy function $E(z)^{-1}$.

Firstly,  we will show the reciprocal energy function $E(z)^{-1}$ is  plurisuperharmonic. 
\begin{lemma}\label{lemma5}
For any $\xi=\xi^\alpha\frac{\p}{\p z^\alpha}\in T_z\mc{T}$ it holds
	\begin{align*}
	\left|\xi^\alpha\frac{\p E(z)}{\p z^\alpha}\right|^2\leq E(z)\cdot \int_{\mc{X}/\mc{T}}\langle \n_{\b{\xi}^\beta\frac{\delta}{\delta\b{z}^\beta}}\p^Vu\wedge\n_{\xi^\alpha\frac{\delta}{\delta z^\alpha}}\b{\p}^Vu\rangle.	
	\end{align*}
\end{lemma}
\begin{proof}
This follows directly from  Theorem \ref{thm2} and Cauchy-Schwarz inequality:
\begin{align*}
	\left|\xi^\alpha\frac{\p E(z)}{\p z^\alpha}\right|^2 &=\left|\xi^\alpha\int_{\mc{X}/\mc{T}}\sqrt{-1}\langle \p^Vu\wedge \n_{\frac{\delta}{\delta z^\alpha}}\b{\p}^Vu\rangle\right|^2\\
	&=\left|\int_{\mc{X}_z}g_{ij}u^i_v\o{\n_{\b{\xi}^\alpha\frac{\delta}{\delta \b{z}^\alpha}}u^j_v}\sqrt{-1}dv\wedge d\b{v}\right|^2\\
	&\leq \int_{\mc{X}_z}g_{ij}u^i_v u^j_{\b{v}}\sqrt{-1}dv\wedge d\b{v}\cdot \int_{\mc{X}_z}g_{ij}\n_{\b{\xi}^\alpha\frac{\delta}{\delta \b{z}^\alpha}}u^i_v\o{\n_{\b{\xi}^\alpha\frac{\delta}{\delta \b{z}^\alpha}}u^j_v}\sqrt{-1}dv\wedge d\b{v}\\
	&=E(z)\cdot \int_{\mc{X}/\mc{T}}\langle \n_{\b{\xi}^\beta\frac{\delta}{\delta\b{z}^\beta}}\p^Vu\wedge\n_{\xi^\alpha\frac{\delta}{\delta z^\alpha}}\b{\p}^Vu\rangle.
\end{align*}
\end{proof}
\begin{thm}\label{thm4}
If $(N,g)$ has non-positive Hermitian sectional curvature, then the function $E(z)^{-1}$ is plurisuperharmonic, i.e. 
\begin{align*}
\sqrt{-1}\p\b{\p}E(z)^{-1}\leq 0.	
\end{align*}
\end{thm}
\begin{proof}
For any vector $\xi=\xi^\alpha\frac{\p}{\p z^\alpha}\in T_z\mc{T}$, 
\begin{align}\label{1.11}
	\frac{\p^2 E(z)^{-1}}{\p z^\alpha\p\b{z}^\beta}\xi^\alpha\b{\xi}^\beta=-\frac{1}{E^2}\left(\frac{\p^2E(z)}{\p z^\alpha\p\b{z}^\beta}-\frac{2}{E}\frac{\p E(z)}{\p z^\alpha}\frac{\p E(z)}{\p\b{z}^\beta}\right)\xi^\alpha\b{\xi}^\beta.
\end{align}
The first term above can be treated using  Theorem \ref{thm3},
\begin{aligns}\label{1.12}
	\frac{\p^2E(z)}{\p z^\alpha\p\b{z}^\beta}\xi^\alpha\b{\xi}^\beta &=		2\int_{\mc{X}/\mc{T}}-R\left(\frac{\p u}{\p v},\xi^\alpha\frac{\delta u}{\delta z^\alpha},\frac{\p u}{\p \b{v}},\b{\xi}^\beta\frac{\delta u}{\delta \b{z}^\beta}\right)\sqrt{-1}\delta v\wedge \delta\b{v}\\
&\quad	+2\int_{\mc{X}/\mc{T}}\langle \n_{\b{\xi}^\beta\frac{\delta}{\delta\b{z}^\beta}}\p^Vu\wedge\n_{\xi^\alpha\frac{\delta}{\delta z^\alpha}}\b{\p}^Vu\rangle\\
&\geq 2\int_{\mc{X}/\mc{T}}\langle \n_{\b{\xi}^\beta\frac{\delta}{\delta\b{z}^\beta}}\p^Vu\wedge\n_{\xi^\alpha\frac{\delta}{\delta z^\alpha}}\b{\p}^Vu\rangle
	\end{aligns}
        by the non-positivity of Hermitian sectional curvature.
Furthermore        Lemma \ref{lemma5} implies  that
\begin{align}\label{1.10}
	\frac{\p^2E(z)}{\p z^\alpha\p\b{z}^\beta}\xi^\alpha\b{\xi}^\beta\geq \frac{2}{E}\left|\xi^\alpha\frac{\p E(z)}{\p z^\alpha}\right|^2=\frac{2}{E}\frac{\p E(z)}{\p z^\alpha}\frac{\p E(z)}{\p\b{z}^\beta}\xi^\alpha\b{\xi}^\beta.
\end{align}
Substituting (\ref{1.10}) into (\ref{1.11}), one has
\begin{align*}
	\frac{\p^2 E(z)^{-1}}{\p z^\alpha\p\b{z}^\beta}\xi^\alpha\b{\xi}^\beta\leq 0.
\end{align*}
Thus 
\begin{align*}
\sqrt{-1}\p\b{\p}E(z)^{-1}=\frac{\p^2E(z)^{-1}}{\p z^\alpha\p\b{z}^\beta}\sqrt{-1}dz^\alpha\wedge d\b{z}^\beta\leq 0.	
\end{align*}
\end{proof}
Next, using Theorem \ref{thm4} we get the following (strict) plurisubharmonicity of logarithmic energy $\log E(z)$.
\begin{thm}\label{thm5}
If $(N,g)$ has non-positive Hermitian sectional curvature, then the logarithmic energy function $\log E(z)$ is plurisubharmonic on Teichm\"uller space $\mc{T}$, i.e. $$\sqrt{-1}\p\b{\p}\log E(z)\geq 0.$$ 	
Moreover, if $(N,g)$ has strictly negative Hermitian sectional curvature and $d(u(z))$ is never zero on $\mc{X}_z$, then $\log E(z)$ is strictly plurisubharmonic, i.e.
$$\sqrt{-1}\p\b{\p}\log E(z)>0.$$
\end{thm}
\begin{proof}
From Theorem \ref{thm4}, we get
\begin{multline}\label{1.13}
\sqrt{-1}\p\b{\p}\log E(z)=-E(z)\sqrt{-1}\p\b{\p}E(z)^{-1}\\+E(z)^{-2}\sqrt{-1}\p E(z)\wedge \b{\p}E(z)\geq 0,	
\end{multline}
which yields the plurisubharmonicity of $\log E(z)$.
To prove the strict
plurisubharmonicity we let
$\xi=\xi^\alpha\frac{\p}{\p z^\alpha}\in T_z\mc{T}$ such that
\begin{align*}
	\frac{\p^2\log E(z)}{\p z^\alpha\p\b{z}^\beta}\xi^\alpha\b{\xi}^\beta=0.
\end{align*}
Then, in view of (\ref{1.12}-\ref{1.13}), 
\begin{align*}
R\left(\frac{\p u}{\p v},\xi^\alpha\frac{\delta u}{\delta z^\alpha},\frac{\p u}{\p \b{v}},\b{\xi}^\beta\frac{\delta u}{\delta \b{z}^\beta}\right)=0,\quad 	\n_{\xi^\alpha\frac{\delta}{\delta z^\alpha}}\b{\p}^Vu=0.
\end{align*}
If $(N,g)$ has strictly negative Hermitian sectional curvature, then
\begin{align}\label{1.9}
\frac{\p u}{\p v}\wedge \xi^\alpha\frac{\delta u}{\delta z^\alpha}=0,\quad \n_{\xi^\alpha\frac{\delta}{\delta z^\alpha}}\b{\p}^Vu=0.	
\end{align}
Since 
\begin{align*}
d(u(z))=u^i_v dv\otimes \frac{\p}{\p x^i}+\o{u^i_v}	d\b{v}\otimes \frac{\p}{\p x^i}
\end{align*}
is never zero on $\mc{X}_z$, so $u^i_v$ is also never zero. From the first equation of (\ref{1.9}), there exists a vector filed $W=W^v\frac{\p}{\p v}\in A^0(\mc{X}_z, T\mc{X}_z)$ such that
\begin{align*}
\xi^\alpha\frac{\delta u^i}{\delta z^\alpha}=W^vu^i_v.	
\end{align*}
The second equation of (\ref{1.9}) is
\begin{align*}
0 &=	\n_{\xi^\alpha\frac{\delta}{\delta z^\alpha}}\b{\p}^Vu\\
&=\xi^\alpha\left(\n_{\frac{\delta }{\delta z^\alpha}}u^i_{\b{v}}\right)\delta\b{v}\otimes \frac{\p}{\p x^i}\\
&=\xi^\alpha\left(\n_{\b{v}}\frac{\delta u^i}{\delta z^\alpha}-A^v_{\alpha\b{v}}u^i_v\right)\delta\b{v}\otimes \frac{\p}{\p x^i}\\
&=\left(\n_{\b{v}}W^vu^i_v-\xi^\alpha A^v_{\alpha\b{v}}u^i_v\right)\delta\b{v}\otimes \frac{\p}{\p x^i}\\
&=\left(\p_{\b{v}} W^v-\xi^\alpha A^v_{\alpha\b{v}}\right)u^i_v\delta\b{v}\otimes \frac{\p}{\p x^i},
\end{align*}
where the last equality follows from harmonic equation $\n_{\b{v}}u^i_v=0$. Thus
\begin{align*}
\xi^\alpha A^v_{\alpha\b{v}}d\b{v}\otimes \frac{\p}{\p v}=\b{\p}W\in A^{0,1}(\mc{X}_z,T\mc{X}_z).	
\end{align*}
This implies that 
\begin{align*}
\rho\left(\xi^\alpha\frac{\p}{\p z^\alpha}\right)=\left[\xi^\alpha A^v_{\alpha\b{v}}d\b{v}\otimes \frac{\p}{\p v}\right]=[\b{\p}W]=0\in H^{1}(\mc{X}_z,T\mc{X}_z).
\end{align*}
Since $\rho: T_z\mc{T}\to H^{1}(\mc{X}_z,T\mc{X}_z)$ is injective, so
$\xi=0$. This proves the
strict plurisubharmonicity.
\end{proof}

The following result was obtained by D. Toledo \cite[Theorem 1, 3]{Toledo}.
\begin{cor}[{\cite[Theorem 1, 3]{Toledo}}]
If $(N,g)$ has non-positive Hermitian sectional curvature, then the energy function $E(z)$ is plurisubharmonic on Teichm\"uller space $\mc{T}$. 
Moreover, if $(N,g)$ has strictly negative Hermitian sectional curvature and $d(u(z))$ is never zero on $\mc{X}_z$, then $ E(z)$ is strictly plurisubharmonic.
\end{cor}
\begin{proof}
Note that
	\begin{align*}
	\sqrt{-1}\p\b{\p}E(z)&=E(z)\sqrt{-1}\p\b{\p}\log E(z)+E(z)^{-1}\sqrt{-1}\p E(z)\wedge \b{\p}E(z)	\\
	&\geq E(z)\sqrt{-1}\p\b{\p}\log E(z).
	\end{align*}
Our claim follows immediately from Theorem \ref{thm5}.
\end{proof}

We give another application of our results
on the variation of
the energy function  in the context
of Hitchin representations. Let $\Gamma=\pi_1(\Sigma)$
be the fundamental group of a closed surface $\Sigma$ of genus $g$.
Let $G$ be a real semisimple Lie group and consider the space
of all reductive representations  $\rho: \Gamma\to G$ 
of $\Gamma$ in $G$ modulo the conjugations by elements in $G$. 
It can be identified with subsets in $G^{2g-2}$ modulo  the
diagonal action of $G$.  When $G$ is a split real form of a complex semisimple 
Lie group there is a distinguished component \cite{Hitchin} called a Hitchin component.
Given any reductive representation $\rho$ of $\Gamma$
 and given a hyperbolic  structure on $\Sigma$, i.e., given
a point $z$ in the Teichm\"u{}ller space $\mathcal T$, there is
a $\rho(\Gamma)$-equivariant  harmonic map $u: \mathbb H^2\to G/K$  from
the hyperbolic plane $\mathbb H^2$ to the Riemannian symmetric space $G/K$, the 
map is unique up to the action of $G$. In particular
the energy function $E_\rho(z)= E(u)=\int_{\mathcal X_z} |du|^2$
is well-defined.
When $G$ is  $SL(n, \mathbb R)$ it is conjectured
by Labourie that for each element $\rho$ in
the Hitchin component there is a unique minimizing point
of $E_\rho(z)$
 in the Teichm\"u{}ller space $\mathcal T$.
Recall \cite{Sampson} that the Riemannian
symmetric space has non-positive Hermitian curvature.
We have thus
\begin{cor} Let $\rho$ be a reductive representation
of $\Gamma$ in $G$. The energy function
 $E_\rho(z)$ is plurisubharmonic on $\mathcal T$.
\end{cor}
It might be interesting to pursue the study of Labourie's conjecture
using our variational formulas.
\section{Energy functions and potentials of Weil-Petersson metric}\label{sec3}

We assume in this section that $N$ is a complex  manifold
with a Hermitian metric $h$. It turns out that in this case
there is a close relation between the second variation of
the energy of $u(z): \mathcal X_z\to N$ and
Weil-Petersson metric.

Let
$\{s^i\}_{1\leq i\leq \dim_{\mb{C}} N}$ be a local holomorphic
coordinates system of $N$.
The Riemannian metric $g=\Re\ h$ is
$$ g=g_{i\bar j}(ds^i\otimes d\bar s^j +d\bar s^j\otimes ds^i)$$ where
$$g_{i\bar j}=g \left(\frac{\p}{\p s^i}, \frac{\p}{\p \bar s^j}\right).$$
The associated two form $\omega=-\Im\ h$
is a two form so that $h=g-\sqrt{-1} \omega$, and
$$g_{jk}=g_{\bar j\bar k}=0,\quad g_{j\bar k}=g_{\bar k j},\quad  g_{\bar j k}=\bar g_{j\bar k}.$$
For any  smooth map  $u: (\mc{X}_z,\Phi_z)\to (N,g)$, 
$$du=u^i_v dv\otimes \frac{\p}{\p s^i} + u^i_{\bar v} d\bar v\otimes \frac{\partial}{\partial s^i}+\o{ u^j_{\b{v}} }dv\otimes \frac{\p}{\p \bar s^j}+ \o{u^j_{v} }d\bar v\otimes \frac{\p}{\p \bar s^j}\in A^1(\mc{X}_z, u^*TN).$$
Hence 
\begin{align}\label{1.22}
	|du|^2=\phi^{v\bar v}g_{i\bar j}(u^i_v \overline{u^j_v}+ u^i_{\bar v}\overline{u^j_{\bar v}}  +\o{u^j_{\b{v}}}u^i_{\b{v}}+ \o{u^j_v}u^i_v)=2\phi^{v\bar v}g_{i\bar j}(u^i_v \overline{u^j_v}+ u^i_{\bar v}\overline{u^j_{\bar v}}).
\end{align}
So the energy is given by
\begin{align}\label{1.23}
	E(u)=\int_{\mc{X}_z} \frac{1}{2}|du|^2 \sqrt{-1}\phi_{v\bar v}dv\wedge d\bar v=\int_{\mc{X}_z} g_{i\bar j}(u^i_v \overline{u^j_v}+ u^i_{\bar v}\overline{u^j_{\bar v}}) \sqrt{-1}dv\wedge d\bar v.
\end{align}
Now we assume that $u:\mc{X}\to N$ is a smooth map such that each
$u(z):\mc{X}_z\to N$ is a harmonic map, and $u(z_0)$ is a holomorphic
map (it is obviously a harmonic map by harmonic equation
(\ref{harmonic})). For notational convinience we write $z_0=o$.
Then $E(z)=E(u(z))$ is a smooth map on Teichm\"uller map $\mc{T}$ and
from Theorem \ref{thm1} we have
  \begin{align*}
  \frac{\p E(z)}{\p z^\alpha} &=-\langle A_\alpha, du\rangle\\
  &=-\langle A^v_{\alpha\b{v}}u^i_vd\b{v}\otimes\frac{\p}{\p s^i}+A^v_{\alpha\b{v}}\o{u^j_{\b{v}}}d\b{v}\otimes\frac{\p}{\p \b{s}^j}, du \rangle\\
  &=-2\int_{\mc{X}_z}g_{i\b{j}}u^i_v\o{u^j_{\b{v}}}A^v_{\alpha\b{v}}\sqrt{-1}dv\wedge d\b{v}.	
  \end{align*}
Evaluating at $o\in \mc{T}$ and using $u(o)$ is holomorphic we get
 \begin{align}\label{1.20}
 \frac{\p E(z)}{\p z^\alpha}|_{z=o}=0.
 \end{align}
The second variation of energy at the point $o\in\mc{T}$ is
\begin{aligns}\label{1.15}
\frac{\p^2E(z)}{\p z^\alpha\p\b{z}^\beta}|_{o}&=-2\int_{\mc{X}_{o}}g_{i\b{j}}u^i_v\frac{\p}{\p\b{z}^\beta}\o{u^j_{\b{v}}}A^v_{\alpha\b{v}}\sqrt{-1}dv
\wedge d\b{v}	\\
&=-2\int_{\mc{X}_{o}}g_{i\b{j}}u^i_v\o{\n_{\frac{\delta}{\delta z^\beta}}u^j_{\b{v}}}A^v_{\alpha\b{v}}\sqrt{-1}dv\wedge d\b{v}\\
&=-2\int_{\mc{X}_{o}}g_{i\b{j}}u^i_v\o{\n_{\b{v}}\frac{\delta u^j}{\delta z^\beta}-A^{v}_{\beta\b{v}}u^j_v}A^v_{\alpha\b{v}}\sqrt{-1}dv\wedge d\b{v}\\
&=2\int_{\mc{X}_{o}}g_{i\b{j}}u^i_v\o{u^j_v}A^v_{\alpha\b{v}}\o{A^v_{\beta\b{v}}}\sqrt{-1}dv\wedge d\b{v}\\
&\quad+2\int_{\mc{X}_{o}}g_{i\b{j}}\n_v(u^i_v)A^v_{\alpha\b{v}}\o{\frac{\delta u^j}{\delta z^\beta}}\sqrt{-1}dv\wedge d\b{v},
\end{aligns}
where the first equality follows from the holomorphicity of $u(o)$, the second equality follows from harmonic equation (\ref{harmonic1}) and the definition of horizontal subbundle (\ref{HV}), the third equality holds by (\ref{1.14}), and the last equality holds by Stokes' theorem and Lemma \ref{lemma1} (i),
\begin{align*}
\n_v A^v_{\alpha\b{v}}=\n_v(A_{\alpha\b{v}\b{v}}\phi^{\b{v}v})=\p_v A_{\alpha\b{v}\b{v}}\phi^{\b{v}v}=0.	
\end{align*}
Here
\begin{align*}
	\n_v u^i_v=\p_v u^i_v+\Gamma^i_{kl}u^k_v u^l_v-\phi_v u^i_v.
\end{align*}
Since $u(o):\mc{X}_z\to N$ is holomorphic,  
\begin{align*}
d(u(o))=u^i_v(o)dv\otimes\frac{\p}{\p x^i}\in A^0(\mc{X}_z,T^*\mc{X}_z\otimes u(o)^*TN) 	
\end{align*}
Let $\n$ denote the natural connection on the bundle $T^*\mc{X}_z\otimes u(o)^*TN$ induced from the Chern connection of $(T^*\mc{X}_z, e^{-\phi})$ and the pullback of Levi-Civita connection $(N,g)$. By conjugation, we also can get a connection on $\o{T^*\mc{X}_z}\otimes u(o)^*TN$, we also denote it by $\n$. 
Then 
\begin{aligns}\label{1.36}
\n d(u(o))&=\left(\p_v u^i_v(o)+\Gamma^i_{kl}u^k_v(o) u^l_v(o)-\phi_v u^i_v(o)\right)	dv\otimes dv\otimes \frac{\p}{\p x^i}\\
&=(\n_vu^i_v)(o)dv\otimes dv\otimes\frac{\p}{\p x^i}.
\end{aligns}
Now we assume that
$u(z):\mc{X}_0\to N$ is totally geodesic
(see e.g. \cite[Definition 1.2.1]{Xin}), i.e.
\begin{align}
\n d(u(0))\equiv 0	
\end{align}
The equation  (\ref{1.15}) becomes
\begin{align}\label{1.17}
	\frac{\p^2E(z)}{\p z^\alpha\p\b{z}^\beta}|_{z=o}=2\int_{\mc{X}_{o}}g_{i\b{j}}u^i_v\o{u^j_v}A^v_{\alpha\b{v}}\o{A^v_{\beta\b{v}}}\sqrt{-1}dv\wedge d\b{v}.
\end{align}
By the harmonic equation $\n_{\b{v}}u^i_v=0$ and the assumption $\n_{v}u^i_v=0$ on $\mc{X}_{o}$, 
\begin{align*}
\n_v(g_{i\b{j}}u^i_v\o{u^j_v}\phi^{v\b{v}})=\frac{\p}{\p v}(g_{i\b{j}}u^i_v\o{u^j_v}\phi^{v\b{v}})=g_{i\b{j}}(\n_v u^i_v\o{u^j_v}+u^i_v\o{\n_{\b{v}}u^j_v})\phi^{v\b{v}}=0.
\end{align*}
This implies that  $(g_{i\b{j}}u^i_v\o{u^j_v}\phi^{v\b{v}})$ is a constant on $\mc{X}_{o}$ and it equals  
\begin{align}\label{1.16}
g_{i\b{j}}u^i_v\o{u^j_v}\phi^{v\b{v}}(o)=\frac{\int_{\mc{X}_{o}}(g_{i\b{j}}u^i_v\o{u^j_v}\phi^{v\b{v}})\sqrt{-1}\phi_{v\b{v}}dv\wedge d\b{v}}{\int_{\mc{X}_{o}}\sqrt{-1}\phi_{v\b{v}}dv\wedge d\b{v}}	=\frac{E(o)}{2\pi(2g-2)}.
\end{align}
Substituting (\ref{1.16}) into (\ref{1.17}), one has
\begin{aligns}\label{1.18}
	\frac{\p^2E(z)}{\p z^\alpha\p\b{z}^\beta}|_{z=o}&=\frac{E(o)}{2\pi(g-1)}\int_{\mc{X}_{o}}A^v_{\alpha\b{v}}\o{A^v_{\beta\b{v}}}\sqrt{-1}\phi_{v\b{v}}dv\wedge d\b{v}\\
	&=\frac{E(o)}{2\pi(g-1)}G_{\alpha\b{\beta}}(o),
\end{aligns}
see (\ref{1.19}) for the definition of $G_{\alpha\b{\beta}}(z)$. 
By (\ref{1.20}), the first variation of energy at $o$ vanishes, so the second variation of $\log E$ satisfies
\begin{align}\label{1.21}
	\sqrt{-1}\p\b{\p}\log E(z)|_{z=o}=\frac{1}{E_0}\sqrt{-1}\p\b{\p}E(z)|_{z=o}=\frac{1}{2\pi(g-1)}\omega_{WP}.
\end{align}

Namely we have
\begin{thm}\label{thm6}
If $u(o)$ is  holomorphic (resp. anti-holomorphic) and totally geodesic
on $\mc{X}_{o}$, then 
	\begin{align*}
	\sqrt{-1}\p\b{\p}\log E(z)|_{z=o}=\frac{\omega_{WP}}{2\pi(g-1)}.	
	\end{align*}
\end{thm}
Specifying to the case
when $N$ is also a Riemann surface
we  obtain Fischer and Tromba's theorem; see
\cite{FT} and \cite[Corollary 5.8]{Wolf1}.
\begin{cor}[{\cite[Theorem 2.6]{FT}}]
If $u(o)=Id: (\mc{X}_{o}, \Phi_{o})\to (\mc{X}_{o}, \Phi_{o)}$ is identity, then  	
\begin{align*}
\sqrt{-1}\p\b{\p}E(z)|_{z=o}=2\omega_{WP}.	
\end{align*}
\end{cor}
\begin{proof}
In this case, $u(o)$ is holomorphic, $u^i_v(o)=\delta^i_v$ and 
\begin{align*}
\Gamma^i_{jk}=\p_v\log \phi_{v\b{v}}=\phi_v.	
\end{align*}
So 
\begin{align*}
	(\n_vu^i_v)(o)=(\p_v u^i_v+\Gamma^i_{kl}u^k_v u^l_v-\phi_v u^i_v)(o)=0
\end{align*}
and 
\begin{align*}
E(o)=	\int_{\mc{X}_{o}}  \sqrt{-1}\phi_{v\b{v}}dv\wedge d\bar v=2\pi(2g-2).
\end{align*}
So the identity
(\ref{1.21}) becomes
$$
\sqrt{-1}\p\b{\p}E(z)|_{z=o}
=
{E_0}\sqrt{-1}\p\b{\p}\log E(z)|_{z=o}=2\omega_{WP}
$$
completing the proof.
\end{proof}
We may also 
apply our result above, as in Section 3, to the
energy function related to a reductive
representation
$\rho: \Gamma=\pi_1(\Sigma) 
\to G$.
 Recall the defintion of the energy 
 function $E_\rho(z)$ for a general element $z$ in the Teichm\"u{}ller
 space in Section 3.
Now let $G$ be a Hermitian semisimple Lie group with
$G/K$ a non-compact Hermitian symmetric space.
Let $\rho: PSL(2, \mathbb R)\to G$ be a fixed representation
 with the induced totally geodesic map 
$\mathbb H^2=PSL(2, \mathbb R)/SO(2) \to G/K$ being
 holomorphic. Let $\mathcal X_o=\mathbb H^2/\Gamma_o $ be
fixed Riemann surface with $\Gamma_o $ a representation of $\Gamma$ in
 $PSL(2, \mathbb R)$.
 The representation  $\rho$ then defines also a representation
 of $\Gamma$, also denoted by $\rho$, i.e. $\rho: \Gamma\to \Gamma_0\subset  
 PSL(2, \mathbb R)\stackrel{\rho}{\to} G$.
 The (lifted) $\rho(\Gamma)$-equivariant map $u(z)$
 for $z=o$ is then holomorphic and totally geodesic
$u(o): \mathbb H^2=PSL(2, \mathbb R)/SO(2) \to G/K$.
We can compute the second variation
of $E_\rho(z)$ at  $z=o$.
\begin{cor} Let $\rho$ be the reductive representation
of $\pi_1(\mathcal X_0)$
in  $G$ obtained from a representation
of $PSL(2, \mathbb R)$ in $G$ with the totally geodesic map
$\mathbb H^2\to G/K$ being holomorphic. Then the second variation
of $E_\rho(z)$ at $z=o$ is 
	\begin{align*}
	\sqrt{-1}\p\b{\p}\log E_\rho(z)|_{z=o}=\frac{\omega_{WP}}{2\pi(g-1)}.	
	\end{align*}
\end{cor}

\section{The second variation of the energy of $u(z): M\to \mathcal
  X_z$ and Weil-Petterson metric}
As we explained in the introduction we may also
consider harmonic maps  $u(z): (M, \omega_g)\to \mathcal X_z$; see
\cite{KWZ}
and references therein.
We assume further that $(M, \omega_g)$ is a compact K\"ahler manifold, i.e. $\omega_g$ is a closed and positive $(1,1)$-form. Let $\{s^i\}_{1\leq i\leq n}$ denote  local  coordinates of $M$, $n=\dim_{\mb{C}} M$. Locally, $\omega_g$ can be expressed as
\begin{align*}
\omega_g=\sqrt{-1}g_{i\b{j}}ds^i\wedge d\b{s}^j
\end{align*}
for some positive definite hermitian matrix $(g_{i\b{j}})$. The associated Riemannian metric $g$ is given by 
\begin{align*}
g=g_{i\b{j}}(ds^i\otimes d\b{s}^j+d\b{s}^j\otimes ds^i).	
\end{align*}
For any smooth map  $u(z): (M, g)\to (\mc{X}_z,\Phi_z)$, $du$ is the section of bundle $T^*M\otimes u^*T_{\mb{C}}\mc{X}_z$, for which there is an induced metric $g^*\otimes \Phi_z$ from $(M^n,g)$ and $(\mc{X}_z,\Phi_z)$.
 Let $\{v\}$ denote the holomorphic coordinates of Riemann surface $\mc{X}_z$. In the same way as in (\ref{1.22}), (\ref{1.23}), one has
\begin{align}\label{1.27}
|du|^2=2g^{\b{j}i}(u_i^v\o{u^v_j}+u^v_{\b{j}}\o{u^v_{\b{i}}})\phi_{v\b{v}}
\end{align}
and the energy is given by
\begin{align*}
E(u)=\frac{1}{2}\int_M |du|^2d\mu_g=\int_M 	g^{\b{j}i}(u_i^v\o{u^v_j}+u^v_{\b{j}}\o{u^v_{\b{i}}})\phi_{v\b{v}}d\mu_g.
\end{align*}
Here $$d\mu_g=\frac{\omega_g^n}{n!}$$ denotes  Riemannian volume form determined by $g$. The harmonic equation is 
\begin{align}\label{harmonic2}
g^{\b{j}i}\n_i u^{v}_{\b{j}}=g^{\b{j}i}(\p_i u^v_j+\phi_vu^v_i u^v_{\b{j}})=0,
\end{align}
see e.g. \cite[(1.20)]{KWZ}.
We assume that $u:M\to \mc{X}$ is a smooth map such that $u(z): M\to
\mc{X}_z$ is a harmonic map and we put $E(z):=E(u(z))$
the energy function on Teichm\"uller space $\mc{T}$.
Similar to (\ref{A-a-1}) we define (with some abuse of notation)
\begin{align*}
A_{\alpha}=A_{\alpha\b{v}\b{v}}\o{u^v_j}\phi^{v\b{v}}d\b{s}^j\otimes \frac{\p}{\p v}+A_{\alpha\b{v}\b{v}}\o{u^v_{\b{i}}}\phi^{v\b{v}}ds^i\otimes \frac{\p}{\p v}\in A^1(M, u^*T\mc{X}_z);
\end{align*}
Let $\Delta=\n\n^*+\n^*\n$ be the Hodge-Laplace operator on $A^{\ell}(M, u^*T\mc{X}_z)$ (see e.g. \cite[Subsection 1.2]{KWZ}), and set 
\begin{align*}
\mc{L}=\Delta+\frac{1}{2}|du|^2,\quad
	\mc{G}=2g^{\b{j}i}\phi_{v\b{v}}u^v_i u^v_{\b{j}} \frac{\p}{\p v}\otimes d\b{v}\in \text{Hom}(u^*\o{T\mc{X}_z},u^*T\mc{X}_z).	
\end{align*}

\begin{thm}[{\cite[Theorem 0.5, 0.6]{KWZ}}]
	The first and the second variation of the energy are given by
	\begin{align}\label{1.24}
		\frac{\p E(z)}{\p z^\alpha}=\langle A_\alpha, du\rangle=2\int_M A_{\alpha\b{v}\b{v}}\o{u^v_i}\o{u^v_{\b{j}}}g^{\b{j}i}d\mu_g
	\end{align}
and
	\begin{multline}\label{1.25}
	\frac{\p^2E(z)}{\p z^{\alpha}\p\b{z}^{\beta}}=\frac{1}{2}\int_M c(\phi)_{\alpha\b{\beta}}|du|^2d\mu_g\\+\langle(Id-\n\left(\mc{L}-\mc{G}\mc{L}^{-1}\o{\mc{G}}\right)^{-1}\n^*)A_{\alpha},A_{\beta}\rangle.	
	\end{multline}
\end{thm}
Now we assume that at $o\in \mc{T}$ the map $u(o)$ is holomorphic. It satisfies the harmonic equation (\ref{harmonic2}) automatically. By (\ref{1.24}), one has
\begin{align}\label{1.32}
	\frac{\p E(z)}{\p z^\alpha}|_{z=o}=0.
\end{align}
We recall \cite[(1.22)]{KWZ} that
\begin{aligns}\label{1.26}
\n^*A_{\alpha} &=\left(-g^{\b{j}i}\n_i(A_{\alpha\b{v}\b{v}}\o{u^v_j}\phi^{v\b{v}})-g^{\b{j}i}	\n_{\b{j}}(A_{\alpha\b{v}\b{v}}\o{u^v_{\b{i}}}\phi^{v\b{v}})\right)\frac{\p}{\p v}\\
&=\left(-g^{\b{j}i}A_{\alpha\b{v}\b{v}}\o{\n_{\b{i}}u^v_j}-g^{\b{j}i}u^v_i\p_v(A_{\alpha\b{v}\b{v}})\o{u^v_j}-g^{\b{j}i}	A_{\alpha\b{v}\b{v}}\o{\n_{j}u^v_{\b{i}}}\right)\phi^{v\b{v}}\frac{\p}{\p v}\\
&=0,
\end{aligns}
where the second equality holds since $u^v_{\b{i}}=0$, the third
equality follows from the harmonic equation (\ref{harmonic2}) and
Lemma \ref{lemma1} (i). Substituting (\ref{1.26}) into (\ref{1.25}) we
find
\begin{align}\label{1.28}
	\frac{\p^2E(z)}{\p z^{\alpha}\p\b{z}^{\beta}}|_{z=o}=\frac{1}{2}\int_M c(\phi)_{\alpha\b{\beta}}|du|^2d\mu_g+\langle A_\alpha, A_\beta\rangle.
\end{align}
\begin{lemma}\label{lemma6}
The following identity holds for any smooth real two form $\alpha$ on $\mc{X}_{o}$
\begin{align*}
\int_M u^*\alpha\wedge \omega_g^{n-1}=\frac{\deg_{\omega_g}(u^*K_{\mc{X}_{o}})}{2g-2}\int_{\mc{X}_{o}}\alpha,
\end{align*}
where 
	$$\deg_{\omega_g}(u^*K_{\mc{X}_{o}})=\int_M u^*c_1(K_{\mc{X}_{o}})\wedge \omega_g^{n-1} .$$
\end{lemma}
\begin{proof}
  Let $\omega_o$
  be the area form on $\mc{X}_o$ such that $\int_{\mathcal
    X_o}\omega_o=c_1(K_{\mathcal X_0})[{\mathcal X_0}]=2g-2$.
  Then $  H^2(\mc{X}_{o}, \mb{R}) = \mb{R}\omega_0$
and  we need only to check the identity for
$  \omega_0$.
We have
\begin{align*}
  \int_M u^*\omega_o\wedge \omega^{n-1}
            	&=(u^*[\omega_o][\omega]^{n-1})[M]\\
	&=(u^*c_1(K_{\mc{X}_{o}})[\omega]^{n-1})[M]\\
	&=\frac{\deg_{\omega_g}(u^*K_{\mc{X}_{o}})}{2g-2}\int_{\mc{X}_{o}}\omega_0.
\end{align*}
\end{proof}
By (\ref{1.27}) and holomorphicity of $u(o)$,  the first term in the RHS of 
(\ref{1.28}) is 
\begin{aligns}\label{1.30}
	\frac{1}{2}\int_M c(\phi)_{\alpha\b{\beta}}|du|^2d\mu_g &=\int_M c(\phi)_{\alpha\b{\beta}}(g^{i\b{j}}\phi_{v\b{v}}u^v_i\o{u^v_j})\frac{\omega_g^n}{n!}\\
&=\int_M u^*(c(\phi)_{\alpha\b{\beta}}\sqrt{-1}\phi_{v\b{v}}dv\wedge d\b{v})\wedge\frac{\omega_g^{n-1}}{(n-1)!}\\
&=\frac{1}{(n-1)!}\frac{\deg_{\omega_g}(u^*K_{\mc{X}_{o}})}{2g-2}G_{\alpha\b{\beta}}, 
\end{aligns}
where the last equality follows from Lemma \ref{lemma6} and (\ref{1.29}), the second equality follows from holomorphicity of $u(o)$ and the following
elementary fact that
\begin{aligns}
\label{trace}
n \alpha\wedge \omega^{n-1}_g=(tr_{\omega_g}\alpha)\omega^n_g,
\end{aligns}
for 
any $(1,1)$-form $\alpha=\sqrt{-1}\alpha_{i\b{j}}ds^i\wedge d\b{s}^j$ with
 $tr_{\omega_g}\alpha:=g^{i\b{j}}\alpha_{i\b{j}}$, $\omega_g=\sqrt{-1}g_{i\b{j}}ds^i\wedge d\b{s}^j$.	

Similarly, by  (\ref{trace}) the second term in the RHS of (\ref{1.28}) is 
\begin{aligns}\label{1.31}
\langle A_{\alpha}, A_{\beta}\rangle&=\int_M (A^v_{\alpha\b{v}}\o{A^v_{\beta\b{v}}}\o{u^v_j}u^v_i g^{i\b{j}}\phi_{v\b{v}})\frac{\omega^n_g}{n!}\\
&=\int_M u^*(A^v_{\alpha\b{v}}\o{A^v_{\beta\b{v}}} \sqrt{-1}\phi_{v\b{v}}dv\wedge d\b{v})\wedge \frac{\omega^{n-1}_g}{(n-1)!}\\	
&=\frac{1}{(n-1)!}\frac{\deg_{\omega_g}(u^*K_{\mc{X}_{o}})}{2g-2}G_{\alpha\b{\beta}}.
\end{aligns}
Substituting (\ref{1.30}) and (\ref{1.31}) into (\ref{1.28}) we have
\begin{align}\label{1.33}
	\frac{\p^2E(z)}{\p z^{\alpha}\p\b{z}^{\beta}}|_{z=o}=\frac{1}{(n-1)!}\frac{\deg_{\omega_g}(u^*K_{\mc{X}_{o}})}{g-1}G_{\alpha\b{\beta}}.
\end{align}
The energy for $u(o)$ is now
\begin{aligns}\label{1.34}
E(o)&=\int_M g^{\b{j}i}u^v_i\o{u^v_j}\phi_{v\b{v}}\frac{\omega^n_g}{n!}\\
&=\int_M u^*(\sqrt{-1}\phi_{v\b{v}}dv\wedge d\b{v})\wedge \frac{\omega^{n-1}_g}{(n-1)!}\\
&=2\pi	\int_M u^*c_1(K_{\mc{X}_{o}})\wedge \frac{\omega^{n-1}_g}{(n-1)!}\\
&=\frac{2\pi}{(n-1)!}\deg_{\omega_g}(u^*K_{\mc{X}_{o}}).
\end{aligns}
Therefore the second variation of $\log E(z)$ at $z=o$, in view of
(\ref{1.32})-(\ref{1.33})-(\ref{1.34}) above, is
\begin{aligns}\label{1.35}
\frac{\p^2\log E(z)}{\p z^\alpha\p\b{z}^\beta}|_{z=o}&=\left(\frac{1}{E(z)}\frac{\p^2E(z)}{\p z^\alpha\p\b{z}^\beta}-\frac{1}{E(z)^2}\frac{\p E(z)}{\p z^\alpha}\frac{\p E(z)}{\p\b{z}^\beta}\right)|_{z=o}	\\
&=\frac{1}{2\pi(g-1)}G_{\alpha\b{\beta}}.
\end{aligns}
Similarly, for anti-holomorphic map $u(o)$, we also can get (\ref{1.35}). Thus 
\begin{thm}
If $u(o)$ is a holomorphic or anti-holomorphic map, then 
\begin{align*}
\sqrt{-1}\p\b{\p}\log E(z)|_{z=o}=\frac{\omega_{WP}}{2\pi(g-1)}.	
\end{align*}
\end{thm}
As a corollary, we obtain
\begin{cor}
If $M$ is a Riemann surface,  and   $u(o)$ is holomorphic or anti-holomorphic, then 
\begin{align*}
\sqrt{-1}\p\b{\p}E(z)|_{z=o}=	|\deg u(o)|\cdot 2\omega_{WP},
\end{align*}
where $\deg u(o)$ is the degree of $u(o)$.
\end{cor}
\begin{proof}
If $M$ is a Riemann surface,  from (\ref{1.33}) 
\begin{align*}
	\frac{\p^2E(z)}{\p z^{\alpha}\p\b{z}^{\beta}}|_{z=o}&=\frac{1}{(n-1)!}\frac{|\deg_{\omega_g}(u^*K_{\mc{X}_{o}})|}{g-1}G_{\alpha\b{\beta}}\\
	&=\frac{|\int_M u^*c_1(K_{\mc{X}_{o}})|}{g-1}G_{\alpha\b{\beta}}\\
	&=|\deg u(o)|\cdot 2G_{\alpha\b{\beta}}.
\end{align*}
Thus
	\begin{align*}
\sqrt{-1}\p\b{\p}E(z)|_{z=o}=\frac{\p^2E(z)}{\p z^{\alpha}\p\b{z}^{\beta}}|_{z=o}\sqrt{-1}dz^\alpha\wedge d\b{z}^\beta=	|\deg u(o)|\cdot 2\omega_{WP}.
\end{align*}
\end{proof}
\begin{rem}
In particular, if $u(o)$ is the identity map, then 
$$ \sqrt{-1}\p\b{\p}E(z)|_{z=o}=2\omega_{WP},$$
which was proved by M. Wolf  \cite[Theorem 5.7]{Wolf1}.	
\end{rem}


\begin{thebibliography}{99}

\bibitem{Ahlfors} L. Ahlfors,  {\it Some Remarks on Teichm\"u{}ller's Space of Riemann Surfaces}, {Ann. Math.}, {\bf 74} (1961), 171-191.

\bibitem{AS0} R. Axelsson, G. Schumacher, {\it Geometric approach to the Weil-Petersson symplectic form}, Comment. Math. Helv. {\bf 85}, (2010), 243-257. 

\bibitem{AS} R. Axelsson, G. Schumacher, {\it Variation of geodesic length functions in families of K\"ahler-Einstein manifolds and applications to Teichm\"uller space}, Ann. Acad. Sci. Fenn. Math., {\bf 37} (2012), no. 1, 91-106. 




\bibitem{Chu} T. Chu, {\it The Weil-Petersson metric in the moduli space}, Chinese J. Math. {\bf 4} (1976), 29-51.


\bibitem{EW} J. Eells, J. C. Wood, {\it Restrictions on Harmonic maps of surfaces}, Topology {\bf 15} (1976), 263-266.

\bibitem{EL}
  J. Eells, L. Lemaire, {\it Deformation of metrics and associated harmonic maps}, Patodi Memorial Volume, Geometry and Analysis Tata Inst. Bombay 1980, 33-45.



\bibitem{Wan1} H. Feng, K. Liu and X. Wan,
   {\it Geodesic-Einstein  metrics and nonlinear stabilities}, to appear, Trans. Amer. Math. Soc.


\bibitem{FT} A. E. Fischer, A. J.  Tromba, {\it A new proof that Teichm\"uller space is a cell}, Trans. Amer. Math. Soc. {\bf 303} (1987), no. 1, 257-262. 

\bibitem{Hartman} P. Hartman,
  {\it On homotopic harmonic maps
  }, Canad. J. Math. {\bf 19} (1967), 673-687.

\bibitem{Hitchin} N. Hitchin,
  {\it
    Lie groups and Teichmuller space,
  }
  Topology {\bf 31} (1992), no. 3, 449-473.
  


\bibitem{KWZ}
I. Kim, X. Wan, and G. Zhang,
{\it Plurisubharmonicity and geodesic convexity of energy function on Teichm\"uller space}, arXiv:1809.00255, 2018. 



\bibitem{Labourie}
  F. Labourie, 
 {\it  Cross ratios,  Anosov  representations and
the energy functional
on  Teichm\"u{}ller space,}
Ann. Scient. \'E{}c. Norm. Sup.
{\bf 41} (2008), 439-471

\bibitem{MM}
  M. Micallef and  J. Moore,
  {\it
    Minimal Two-Spheres and the Topology of Manifolds with Positive
    Curvature on Totally Isotropic Two-Planes}
  Ann. Math. {\bf 127} (1988),  199-227.
  

\bibitem{Ono} K. Ono, {\it On the holomorphicity of harmonic maps from compact K\"ahler manifolds to hyperbolic Riemann surfaces}, Proc. Amer. Math. Soc. {\bf 102} (1988), no. 4, 1071-1076. 


\bibitem{Sch} G. Schumacher, {\it Positivity of relative canonical bundles and applications}, Invent. Math. {\bf 190} (2012), no. 1, 1-56.

\bibitem{Sampson1} J. H. Sampson, {\it Some properties and applications of harmonic mappings}, Annales Scient. \'Ecole Normale Sup\'erieure, {\bf 11} (1978), 211-228.

\bibitem{Sampson} J. H. Sampson, {\it  Applications of harmonic maps to K\"ahler geometry}, Contemprory Mathematics, {\bf 49} (1986), 125-133.

\bibitem{Sun} T. Sunada, {\it Rigidity of certain harmonic mappings}, Inventiones of Mathematics, {\bf 51} (1979), 297-307.

\bibitem{Siu} Y. T. Siu, {\it The complex analyticity of harmonic maps and the strong rigidity of compact K\"ahler manifolds}, Ann. of Math. (2) {\bf 112} (1980), 73-111.

\bibitem{SY} Y. T. Siu, S.-T. Yau, {\it Compact K\"ahler manifolds of positive bisectional curvature}, Invent. Math. {\bf 59} (1980), 189-204.

\bibitem{Toledo} 
 D. Toledo,  {\it Hermitian curvature and plurisubharmonicity of energy on Teichm\"uller space}, Geom. Funct. Anal. {\bf 22} (2012), no. 4, 1015-1032.


\bibitem{Tromba0} A. Tromba, {\it Teichm\"uller theory in Riemannian 
   geometry}, Lecture notes prepared by Jochen Denzler, Lectures in 
  Mathematics ETH Z\"urich 
 {Birkh\"auser} Verlag, Basel, 1992. 
  


\bibitem{Tromba2} A. J. Tromba, {\it On a natural algebraic affine connection on the space of almost complex structrues and the curvature of Teichm\"uller space with respect to its Weil-Petersson metric}, Manuscripta Math. {\bf 56} (1986), 475-497.
 

\bibitem{Wolf1} M. Wolf, {\it The Teichm\"u{}ller theory of harmonic maps}, J. Differential Geom. {\bf 29} (1989), 449-479.


\bibitem{Wol1} S. Wolpert, {\it Geodesic length functions and the Nielsen problem},  J. Differential Geom. {\bf 25} (1987), 275-296.




\bibitem{Wol4} S. Wolpert, {\it Noncompleteness of the Weil-Petersson metric for Teichm\"uller space}, Pacific J. Math. {\bf 61} (1975), 573-577.

\bibitem{Wol5} S. Wolpert, {\it Chern forms and the Riemann tensor for the moduli space of curves}, Invent. Math. {\bf 85} (1986), 475-497.

\bibitem{Xin}  Y.  Xin,  {\it Geometry of harmonic maps,} Progress in Nonlinear Differential Equations and their Applications, 23. Birkh\"auser Boston, Inc., Boston, MA, 1996.

\bibitem{Yamada} S. Yamada, {\it Weil-Petersson  convexity of the energy function on classical and universal Teichm\"u{}ller spaces}, J. Differential Geom. {\bf 51} (1999), no. 1, 35-96.



\end{thebibliography}
\end{document}